\documentclass[12pt,a4paper,psamsfonts]{amsart}

\usepackage{fancyhdr,mathtools}
\usepackage{appendix}
\usepackage{amssymb,amscd,amsxtra,calc}
\usepackage{mathrsfs}
\usepackage{cmmib57}
\usepackage{multirow}
\usepackage[all]{xy}
\usepackage{longtable}
\usepackage[colorlinks,linkcolor=blue,anchorcolor=red,citecolor=green,pagebackref,linktocpage]{hyperref}
\usepackage{cite}	
\usepackage{mathabx,epsfig}
\def\acts{\ \rotatebox[origin=c]{-90}{$\circlearrowright$}\ }
\def\racts{\ \rotatebox[origin=c]{90}{$\circlearrowleft$}\ }

\theoremstyle{plain}
    \newtheorem{thm}{Theorem}[section]

       \newtheorem{lemma}[thm]{Lemma}
           \newtheorem{theorem}[thm]{Theorem}
            \newtheorem*{maintheorem}{Main Theorem}
    \newtheorem{proposition}[thm]{Proposition}

\theoremstyle{definition}
 
    \newtheorem{claim}[thm]{Claim}
    \newtheorem{definition}[thm]{Definition}

       \newtheorem{conjecture}[thm]{Conjecture}
    \newtheorem{question}[thm]{Question}

    \newtheorem*{notation*}{Notation and Terminology}
    \newtheorem{remark}[thm]{Remark}
    \newtheorem*{ack}{Acknowledgments}
\theoremstyle{remark}

\newcommand{\Q}{\mathbb{Q}}

\newcommand{\OO}{\mathcal{O}}

\newcommand{\alb}{\operatorname{alb}}

\newcommand{\id}{\operatorname{id}}

\newcommand{\Nef}{\operatorname{Nef}}

\newcommand{\NS}{\operatorname{NS}}

\newcommand{\PE}{\operatorname{PE}}

\newcommand{\SEnd}{\operatorname{SEnd}}

\newcommand{\Supp}{\operatorname{Supp}}

\newcommand{\Codim}{\operatorname{codim}}

\newcommand{\N}{\operatorname{N}}

\newcommand{\Alb}{\operatorname{Alb}}

\newcommand{\Pic}{\operatorname{Pic}}

\newcommand{\alg}{\mathrm{alg}}
\newcommand{\reg}{\mathrm{reg}}

\begin{document}

\title[Kawaguchi-Silverman conjecture]
{Kawaguchi-Silverman conjecture for int-amplified endomorphism}

\author{Sheng Meng, Guolei Zhong}

\address{
\textsc{School of Mathematical Sciences, Shanghai Key Laboratory of PMMP}\endgraf 
\textsc{East China Normal University, Shanghai 200241, China}
}
\email{smeng@math.ecnu.edu.cn}

\address{
\textsc{Center for Complex Geometry, 
	Institute for Basic Science (IBS)}\endgraf
	\textsc{55 Expo-ro, Yuseong-gu, Daejeon, 34126, Republic of Korea
}}
\email{zhongguolei@u.nus.edu, guolei@ibs.re.kr}

\begin{abstract}
Let $X$ be a $\mathbb{Q}$-factorial klt projective variety admitting an int-amplified endomorphism $f$, i.e., the modulus of any eigenvalue of $f^*|_{\NS(X)}$ is greater than $1$. 
We prove Kawaguchi-Silverman conjecture for $f$ and also any other surjective endomorphism of $X$: the first dynamical degree equals the arithmetic degree of any point with Zariski dense orbit.
This generalizes an early result of Kawaguchi and Silverman \cite{KS16b} for the polarized $f$ case, i.e., $f^*|_{\NS(X)}$ is diagonalizable with all eigenvalues of the same modulus greater than $1$.
\end{abstract}



\maketitle
\tableofcontents
\section{Introduction}\label{sec-intro}

We work over an algebraically closed field $\mathbf{k}$ of characteristic zero.

Let $f\colon  X \to X$ be a surjective endomorphism of a normal projective variety $X$. There are two fundamental dynamical invariants that characterize $f$ from the perspectives of topology and arithmetic.

Let \(H\) be an ample divisor of \(X\).
Then the \textit{first dynamical degree} $\delta_f$ is defined as
$$
\delta_f \coloneqq \lim_{m \to +\infty} \left( (f^m)^*H \cdot H^{\dim X - 1} \right)^{1/m}.
$$
This limit, which is independent of the choice of $H$, exists and is equal to the spectral radius of $f^*|_{\NS(X) \otimes_{\mathbb{Z}} \mathbb{C}}$, where $\NS(X)$ is the N\'eron-Severi group of $X$ (see \cite{DS04,DS05}; cf.~\cite{Dan19}).

Suppose that  $\mathbf{k}$ is the field $\overline{\mathbb{Q}}$ of algebraic numbers. 
Let $h_H$ be a Weil height function associated with an ample divisor $H$ of $X$. 
The \textit{arithmetic degree} is defined as a function on $X(\overline{\mathbb{Q}})$:
$$
\alpha_f(x) \coloneqq \lim_{m \to +\infty} \max\{1, h_H(f^m(x))\}^{1/m}. 
$$
This limit, which is also independent of the choice of $H$, exists and equals either $1$ or the modulus of some eigenvalue of $f^*|_{\NS(X) \otimes_{\mathbb{Z}} \mathbb{C}}$ (see \cite[Proposition 12]{KS16a} and \cite[Theorem 3]{KS16b}).

We refer the reader to Definitions \ref{def-dyn} and \ref{def-ari} for a detailed explanation.

\subsection{Kawaguchi-Silverman Conjecture}\label{subsec-ksc}
Our starting point in this paper is to study the \textit{Kawaguchi-Silverman Conjecture} (abbreviated as \textit{KSC}), which was proposed by Kawaguchi and Silverman \cite[Conjecture 6]{KS16b}. 
Originally formulated for dominant rational self-maps, we focus here on its version for endomorphisms.

\begin{conjecture}[Kawaguchi-Silverman Conjecture = KSC]\label{conj-KSC}
Let $f\colon X\to X$ be a surjective endomorphism of a projective variety $X$ over $\overline{\mathbb{Q}}$.
Then 
$$\alpha_f(x) = \delta_f$$ for any $x \in X(\overline{\mathbb{Q}})$ with the orbit $O_f(x) = \{f^n(x)\, |\, n\ge 0\}$ being Zariski dense
in $X$.
\end{conjecture}

\begin{remark}\label{rmk-known-case}
To the best knowledge of ourselves, Conjecture \ref{conj-KSC} is known to hold in the following cases. 
\begin{itemize}
\item[(i)]
$f$ is polarized (see \cite[Theorem 5]{KS16b}).
\item[(ii)]
$X$ is a smooth projective surface (see \cite[Theorem~10]{KS14} and \cite[Theorem~1.3]{MSS18}).
\item[(iii)]
$X$ is a Mori dream space
(e.g.~of Fano type; see \cite[Theorem 4.1]{Mat20a}).
\item[(iv)]
$X$ is an abelian variety
(see \cite[Corollary 32]{KS16b}, \cite[Theorem~2]{Sil17}).
\item[(v)] $X$ is a Hyperk\"ahler variety (see \cite[Theorem~1.2]{LS21}).
\item[(vi)] $X$ is a smooth rationally connected projective variety admitting an int-amplified endomorphism (see \cite{MZ22} and \cite{MY22} for threefolds and \(n\)-folds respectively).
\item[(vii)] $X$ is a smooth projective 3-fold and $\deg (f) > 1$ (see \cite{MZ23b}; cf.~\cite[Prop.~1.6]{LS21}). 
\end{itemize}
We refer to \cite{CLO22,HZ23,JSXZ21,KS14,MW22,Mat20b,Mat24,Ogu24,Xie24} for various other settings of KSC (e.g., the version for dominant rational self-maps). For a comprehensive overview of the current state of Conjecture \ref{conj-KSC}, we also direct the reader to the recent survey paper \cite{Matz23}. 
\end{remark}

Among the various advances towards the KSC, a crucial tool is the \textit{canonical height function} (cf.~\cite{CS93}), used to evaluate the arithmetic degree. 
Let $D \in \NS(X) \otimes_{\mathbb{Z}} \mathbb{R}$ be a divisor satisfying $f^*D \sim qD$ for some $q > 1$.
Then the canonical height associated to \(D\) and \(f\) is defined as
$$
\hat{h}_{D,f}(x) = \lim_{n \to +\infty} \frac{h_D(f^n(x))}{q^n}.
$$
This limit exists and satisfies
$$
\hat{h}_{D,f} \circ f = q \cdot \hat{h}_{D,f}, \quad \textup{and} \quad \hat{h}_{D,f} = h_D + O(1),
$$
where $O(1)$ denotes a bounded function (see \cite[Theorem 1.1]{CS93}). By \cite[Theorem 1.9]{MMS+22}, we can find a nef $\mathbb{R}$-Cartier divisor $D \nequiv 0$ such that $f^*D \sim_{\mathbb{R}} \delta_f D$. 
Thus, if we further assume that $\hat{h}_{D,f}(x) \neq 0$ for any given \(x\) with Zariski dense orbit,  
then we obtain that $\alpha_f(x) \ge \delta_f$ (cf.~\cite[Lemma 3.3]{MSS18}); in particular, KSC holds in this case.

When $D$ is ample, the \textit{Northcott finiteness property} (cf.~\cite[Theorem B.3.2(g)]{HS00} or \cite[Theorem 2.6]{Lan83}) asserts that the set
$$
\{ x \in X(\overline{\mathbb{Q}}) \mid [K(x):\mathbb{Q}] \leq d, \; h_D(x) \leq M \}
$$
is finite for any given $d > 0$ and $M > 0$, where $K(x)$ is the number field defining $x$. This implies that $\hat{h}_{D,f}(x) = 0$ if and only if the orbit $O_f(x)$ is finite. 
When $D$ is $\mathbb{Q}$-Cartier with the Iitaka dimension $\kappa(X,D)>0$, we have $\hat h_{D,f}(x)\neq 0$ if $O_f(x)$ is Zariski dense in $X$ by taking the Iitaka fibration; see \cite[Proposition 3.6]{Mat20a}.

\vskip 2mm
\noindent\textbf{Difficulty.}
In general, $\delta_f$ is not a rational number, so we cannot always expect the existence of a $\mathbb{Q}$-Cartier eigenvector $D$. Moreover, even if such a $\mathbb{Q}$-Cartier divisor exists, it does not necessarily have positive Iitaka dimension. 
As a result, the machinery of canonical heights does not function as effectively as it does in the case of polarized endomorphisms (cf.~\cite[Question 10]{KS16b}).

\subsection{Main result}
A surjective endomorphism $f$ of a normal projective variety $X$ is called \textit{int-amplified} if \( f^*H - H \) is ample for some ample divisor $H$ (cf.~Definition \ref{def-endo}). Alternatively, this is equivalent to having \(\iota_f > 1\), where \(\iota_f\) denotes the smallest modulus of the eigenvalues of \(f^*|_{\NS(X) \otimes_{\mathbb{Z}} \mathbb{C}}\) (see \cite[Theorem 1.1]{Men20}). It is also equivalent to being \textit{\((\dim X)\)-cohomologically hyperbolic}, i.e., the last dynamical degree is strictly larger than the other dynamical degrees (see \cite[Proposition 3.7]{MZ23}).

Int-amplified endomorphisms have gradually emerged as key tools in the study of general endomorphisms, particularly in the reduction of a non-isomorphic endomorphism to a polarized endomorphism. For several applications in different contexts, we refer the reader to \cite{JXZ23} and \cite{MZZ22}.

Zhang and the first author \cite{MZ22} made the initial attempt to prove KSC for int-amplified endomorphisms using the theory of \textit{equivariant minimal model program} (EMMP); 
see \cite{Men20, MZ20} for further details. The program progressed smoothly until a specific obstruction emerged, referred to as \hyperlink{Case TIR}{Case TIR}. Roughly speaking, TIR stands for ``totally invariant ramification'', but it involves additional geometric and dynamical constraints, as detailed below.

\hypertarget{Case TIR}{\vskip 2mm}
\noindent{\textbf{Case TIR}$_n$} (Totally Invariant Ramification case).
Let $X$ be a normal projective variety of dimension $n \ge 1$, which has only $\Q$-factorial Kawamata log terminal (klt) singularities and admits one int-amplified endomorphism.
Let $f\colon X\to X$ be an arbitrary surjective endomorphism.
Moreover, we impose the following conditions.
\begin{itemize}
\item[(A1)]
The anti-Kodaira dimension $\kappa(X,-K_X)=0$; $-K_X$ is nef, whose class is extremal in both the {\it nef cone} $\Nef(X)$ and the {\it pseudo-effective divisors cone} $\PE^1(X)$.
\item[(A2)]
$f^*D = \delta_f D$ for some prime divisor $D\sim_{\Q} -K_X$.
\item[(A3)]
The ramification divisor of $f$ satisfies $\Supp \, R_f = D$.
\item[(A4)]
There is an $f$-equivariant Fano contraction $\tau\colon X\to Y$ with $\delta_f>\delta_{f|_Y}$ ($\ge 1$).
\end{itemize}

\begin{theorem}[{see \cite[Theorem 1.7]{MZ22}}]\label{theorem-tir}
Let $X$ be a $\Q$-factorial klt projective variety admitting an int-amplified endomorphism.
Then we have:
\begin{itemize}
\item[(1)] If $K_X$ is pseudo-effective, then KSC holds for any surjective endomorphism of $X$.
\item[(2)] Suppose that KSC holds for \hyperlink{Case TIR}{Case TIR} 
(for those $f|_{X_i} \colon  X_i \to X_i$ appearing in any equivariant MMP starting from $X$).
Then KSC holds for any (not necessarily int-amplified) surjective endomorphism $f$ of $X$.
\end{itemize}
\end{theorem}

In \cite[Theorem 8.6]{MZ22} and subsequently in \cite[Theorem 6.6]{MMSZ23}, it was shown that \hyperlink{Case TIR}{Case TIR} is not possible when \(\dim X - \dim Y = 1\) or when \(\dim X - \dim Y = 2\) and \(X\) has terminal singularities. Building on this, KSC was proved for any surjective endomorphism of a smooth projective threefold that admits an int-amplified endomorphism; see \cite[Theorem 1.5]{MMSZ23}. As mentioned earlier, this result further contributed to proving KSC for any non-isomorphic surjective endomorphism of a smooth projective threefold without the requirement of int-amplified endomorphisms; see \cite[Theorem 1.7]{MZ23}. The general question regarding \hyperlink{Case TIR}{Case TIR} is posed in \cite[Question 1.8]{MZ22}.

\begin{question}\label{question-tir}
	Does there exist $f\colon  X \to X$ satisfying \hyperlink{Case TIR}{Case TIR} (plus, if necessary, that $X$ is rationally connected as defined below)?
	If such $f$ exists,
	does it satisfy KSC?
\end{question}

In this paper, we provide a complete answer to the second part of Question \ref{question-tir}: KSC holds for \hyperlink{Case TIR}{Case TIR} (cf.~Remark \ref{rem-non-existence-tir}). Thus, we can now state our main result.

\begin{maintheorem}\label{main-theorem-ksc}
Let $X$ be a $\mathbb{Q}$-factorial klt projective variety admitting an int-amplified endomorphism.
Then the Kawaguchi-Silverman conjecture holds for any surjective endomorphism $f$ of $X$.
\end{maintheorem}

We note that our main theorem extends \cite{KS16b} for polarized endomorphisms and \cite{MY22} for smooth rationally connected projective varieties (see Remark \ref{rmk-known-case} (i) and (vi)).

\subsection{Strategy: reducibility and equivariancy}\label{subsec-strategy}
In the remainder of the introduction, we will briefly outline our strategy.

We begin with \hyperlink{Case TIR}{Case TIR}. Our approach involves either finding an alternative EMMP that avoids \hyperlink{Case TIR}{Case TIR} or ensuring that the first Fano contraction in the EMMP does not satisfy \hyperlink{Case TIR}{Case TIR} by verifying that the divisor $D$ is reducible. 
To achieve this, we might need to replace the original variety $X$ with a finite cover that is equivariant with respect to both $f$ and an int-amplified endomorphism. 
Utilizing such a special cover, we provide a relatively simple geometric description of the highly restrictive \hyperlink{Case TIR}{Case TIR}. This approach extends previous works \cite{HN11, MZ19, MZg23, MYY24} to the case of singular vibrations. 
For relevant terminology, we refer to Definition \ref{defn-locally-trivial}, and for a detailed comparison with previous works and a further question, see Remark \ref{rmk-comparison} and Question \ref{ques}, respectively.

Let $X$ be a normal projective variety and $D$ a reduced divisor of $X$.
Denote by $\SEnd(X,D)$ the monoid of surjective endomorphisms $f$ of $X$ with $f^{-1}(D)=D$.

\begin{theorem}[Equivariant toric cover]\label{thm-tir}
Let $X$ be a klt projective variety and $D$ a reduced divisor such that $K_X+D\equiv 0$.
Suppose $\SEnd(X,D)$ contains an int-amplified endomorphism. 
Then for any $f\in \SEnd(X,D)$, there exists a quasi-\'etale cover \(\pi\colon \widehat{X}\to X\) satisfying: 
\begin{enumerate}
\item the pair \((\widehat{X}, \pi^*D)\) admits a splitting toric fibration over an abelian variety \(A\),
\item there exists $\widetilde{f}\in \SEnd(\widehat{X}, \pi^*D)$ such that $\pi\circ\widetilde{f}=f^s\circ \pi$ for some $s>0$, and 
\item $\SEnd(\widehat{X},\pi^*D)$ contains an int-amplified endomorphism.  
\end{enumerate}
\end{theorem}

It is worth noting that the condition \( K_X + D \equiv 0 \) in Theorem \ref{thm-tir} is equivalent to \(X\) admitting an int-amplified endomorphism \(\mathcal{I}\) such that \(\mathcal{I}^{-1}(D) = D\) and \(\mathcal{I}|_{X \setminus D}\) is quasi-\'etale, i.e., \( K_X + D = \mathcal{I}^*(K_X + D) \) (see \cite[Theorem 1.1]{Men20}).

The proof of Theorem \ref{thm-tir} involves three key ingredients: (1) the positivity of a dynamically stable reflexive sheaf, as discussed in Theorem \ref{thm-invariant-sheaf}; (2) the equivariant lifting to a suitable quasi-\'etale cover, detailed in Claim \ref{claim-lift} and Lemma \ref{lem-toric-splitting-bc}; and (3) a structure theorem established by Druel and Lo Bianco \cite[Theorem 6.1]{DLB22} (cf.~Theorem \ref{thm-struc-k}) for complex klt projective varieties with numerically flat and locally free logarithmic differential sheaves.

\begin{remark}\label{rmk-comparison}
We provide several remarks on the literature review of Theorem \ref{thm-tir}:
\begin{enumerate}
\item[(i)]  We briefly explain the previous approach to Theorem \ref{thm-tir}(1) when \(X\) is \textbf{of Fano type and smooth in codimension two}. As first observed in \cite{HN11} and later generalized in \cite{MZ19} and \cite{MZg23}, in this scenario, \(\Omega_X^{[1]}(\textup{log}\,D)\) has vanishing first and second Chern classes. Together with slope semistability, this allows one to conclude by a generalization of Simpson's correspondence; see \cite[Theorem 1.20]{GKP16}. 
\item[(ii)] If the assumption that \(X\) is smooth in codimension two is removed, one can still define the second orbifold \(Q\)-Chern class, as klt singularities are quotient singularities away from a codimension three closed subset (see, e.g., \cite[Section 3]{GKPT19}). However, it is unclear whether the Mumford construction is functorial with respect to an endomorphism \(f\). Additionally, verifying \(f^*\widehat{c}_2(\Omega_X^{[1]}(\textup{log}\,D)) = \widehat{c}_2(\Omega_X^{[1]}(\textup{log}\,D))\) as a multilinear form on \(\textup{N}^1(X)^{\times (n-2)}\) seems to be problematic.
\item[(iii)] In the recent papers \cite{MYY24,MYY24b}, Moraga, Y\'a\~nez, and Yeong obtained a similar result to Theorem \ref{thm-tir} using a completely different approach. Additionally, assuming the numerical flatness of the logarithmic tangent bundle of a smooth projective variety in positive characteristic, Ejiri and Yoshikawa obtained a result similar to \cite[Theorem 1.1]{DLB22}; see \cite[Theorem 1.4]{EY23}.
\item[(iv)] When \(X\) has vanishing augmented irregularity (Definition \ref{defn-irreg}), \cite[Lemma 5.1]{MZ19} shows that there exists a universal cover (which is a toric variety) to which any surjective endomorphism can be lifted. This is primarily due to the finiteness of the algebraic fundamental group \(\pi_1^{\alg}(X_{\reg})\). However, in our relative case, the algebraic fundamental group is no longer finite, and we could not identify a canonical finite cover to which \(f\) lifts. Instead, the model \(\widehat{X}\) in Theorem \ref{thm-tir} depends significantly on both the choice of \(f\) and the int-amplified endomorphism \(\mathcal{I}\). Moreover, the int-amplified endomorphism of \(\widehat{X}\) is derived from another int-amplified endomorphism generated by \(f\) and \(\mathcal{I}\).
\end{enumerate}
\end{remark}

With the assistance of Theorem \ref{thm-tir}, we achieve a geometric structure for \(\widehat{X}\) that is quite favorable. However, this comes at the cost that \(\widehat{X}\) may not necessarily retain the mild singularities of \(X\), particularly its \(\mathbb{Q}\)-factoriality. Therefore, it is necessary to further resolve the singularities in an equivariant manner, as outlined in the following theorem. Note, however, that an equivariant resolution may not always exist; for example, see \cite[Section 7]{MY21}.

\begin{theorem}[Equivariant resolution; see Theorem \ref{thm-equi-lifting-general} for a more general form]\label{thm-equi-lifting}
Let \(X\) be a normal projective variety and \(D\) a reduced divisor.
Let \(\pi\colon (X,D)\to Y\) be a splitting toric fibration over an abelian variety $Y$. 
Then there exists a resolution $\sigma\colon  \widetilde{X}\to X$ such that for any $f\in \SEnd(X,D)$ with $f^*D=qD$, there exists some $\widetilde{f}\in \SEnd(\widetilde{X},\sigma^{-1}(D))$ with $\sigma\circ \widetilde{f}=f^s\circ \sigma$ for some $s>0$.
\end{theorem}

With the help of Theorem \ref{thm-equi-lifting}, we may continue to run EMMP:
 
$$\xymatrix{
\widetilde{X}=\widetilde{X}_1\ar@{-->}[r]^{\widetilde{\pi}_1}& \widetilde{X}_2\ar@{-->}[r]^{\widetilde{\pi}_2} &\cdots\ar@{-->}[r]^{\widetilde{\pi}_{s-1}} &\widetilde{X}_{s}\ar[r]^{\widetilde{\tau}} &\widetilde{Y}
}$$
Here, each \(\widetilde{\pi}_i\) is a birational map, and \(\widetilde{\tau}\) is the first Fano contraction. Note that \(\widetilde{X}_s \to A\) remains a splitting smooth toric fibration (see Remark \ref{rmk-trivial-mmp}). In particular, \(-K_{\widetilde{X}_s}\) is \(\mathbb{Q}\)-linearly equivalent to some reduced divisor with at least two components. Consequently, the irreducible assumption of \hyperlink{Case TIR}{Case TIR} is not satisfied.

\begin{ack}
The authors would like to thank Junyi Xie and De-Qi Zhang for the communication of Section \ref{sec-fib}; the proofs of Lemma \ref{lem-basechange} and Theorem \ref{thm-toric-trivial} are essentially borrowed from their ongoing collaborative work with the first author \cite{MXZ24}. 
The authors would also like to thank Fei Hu, Yohsuke Matsuzawa, and Joseph H. Silverman for the valuable suggestions and comments. 
The first author was supported by Fundamental Research Funds for the Central Universities, 
Science and Technology Commission of Shanghai Municipality (No. 22DZ2229014) and a National Natural Science Fund.
The second author was supported by the Institute for Basic Science (IBS-R032-D1).
\end{ack}

\section{Preliminary}\label{sec-pre}

We exhibit the notation and terminology involved in this paper.

\subsection{Varieties, divisors and singularities}

\begin{definition}\label{not}
Let \(X\) be a normal projective variety of dimension \(n\).
\begin{enumerate}
\item Denote by $\textup{NS}(X)\coloneqq\Pic(X)/\Pic^\circ(X)$ the N\'eron-Severi group.
Let $\textup{N}^1(X)\coloneqq\textup{NS}(X)\otimes_\mathbb{Z}\mathbb{R}$. 
Denote by $\textup{PE}^1(X)$  (resp. \(\textup{Nef}(X)\)) the {\it pseudo-effective cone} (resp. {\it nef cone}) in $\textup{N}^1(X)$. 
We will use these two cones in \hyperlink{Case TIR}{Case TIR}.
\item The symbols $\sim$ (resp.\,$\sim_{\mathbb Q}$, \(\sim_{\mathbb{R}}\),
$\equiv$) denote
the linear equivalence (resp.\,$\mathbb Q$-linear equivalence, \(\mathbb{R}\)-linear equivalence, numerical equivalence) on divisors.
\item  An $r$-cycle $C$ on  $X$ is {\it pseudo-effective} if $C \cdot H_1\cdots\cdot H_{n-r}\ge 0$ for any ample divisors $H_1,\cdots,H_{n-r}$ on \(X\). 
Two $r$-cycles $C_1$ and $C_2$ are said to be {\it weakly numerically equivalent}, denoted as $C_1\equiv_w C_2$  if
$(C_1 - C_2) \cdot L_1\cdots L_{n-r}= 0$ for any 
Cartier divisors $L_1,\cdots,L_{n-r}$. 
Let \(\N_{n-1}(X)\) be the quotient of vector space of Weil \(\mathbb{R}\)-divisors modulo the weak numerical equivalence. 
\item A normal projective variety $X$ is of \textit{Fano type}, if there is an effective Weil $\mathbb{Q}$-divisor $\Delta$ on $X$ such that the pair $(X,\Delta)$ has at worst klt singularities and $-(K_X+\Delta)$ is ample and $\mathbb{Q}$-Cartier.
\end{enumerate}
\end{definition}

We refer to \cite[Section 2]{KM98} for the standard notion and terminology on different kinds of singularities. 
 
\begin{definition}
Let \(f\colon X\to Y\) be a finite surjective morphism between normal projective varieties.
The \textit{ramification divisor \(R_f\) of \(f\)} is defined by the formula 
\begin{align}\label{eq-ramification}
K_X=f^*K_Y+R_f.
\end{align}
We call (\ref{eq-ramification}) the  {\it ramification divisor formula}.
We denote by  \(B_f=f(\Supp R_f)\) the  (reduced) \textit{branch divisor} of \(f\).  
We say that \(f\) is \textit{quasi-\'etale}, if it is \'etale away from a codimension two closed subset; in particular, we have \(B_f=0\) and \(K_X=f^*K_Y\) holds as the equality of Weil divisors.
\end{definition}

\begin{definition}\label{defn-irreg}
Let \(X\) be a normal projective variety, we define the \textit{irregularity} of \(X\) by 
$$q(X)\coloneqq h^1(\widehat{X},\mathcal{O}_{\widehat{X}})$$ where \(\widehat{X}\to X\) is a resolution of singularities.
We define the \textit{augmented irregularity} \(\widetilde{q}(X)\in\mathbb{N}\cup\{\infty\}\) to be the supremum of \(q(X')\) whenever \(X'\) runs over the quasi-\'etale covers of \(X\) (see \cite[Section 2]{NZ10}). 
When \(X\) has only rational singularities (e.g., klt), \(q(X)\) coincides with \(h^1(X,\mathcal{O}_X)\) (see \cite[Definition 5.8 and Theorem 5.22]{KM98} and \cite[Proposition 2.3]{Rei83} or \cite[Lemma 8.1]{Kaw85}). 
\end{definition}

\subsection{Reflexive sheaves}
Given a coherent sheaf \(\mathcal{E}\) on a scheme, there is a dualizing operation 
\(\mathcal{E}^{\vee}\coloneqq \textup{Hom}(\mathcal{E},\mathcal{O}_X)\) and 
there is a natural map \(\mathcal{E}\to\mathcal{E}^{\vee\vee}\) to its double-dual.
If this map is an isomorphism, we say that \(\mathcal{E}\) is \textit{reflexive}.

\begin{definition}
Let \(X\) be a normal projective variety, and \(\mathcal{E}\) a reflexive sheaf on \(X\).
It is known that \(\mathcal{E}\) is uniquely determined on a Zariski open subset of \(X\) whose complement has codimension \(\geq 2\). 
We define 
\begin{enumerate}
\item the \textit{reflexive tensor power} $\mathcal{E}^{[\otimes m]}$ by $\mathcal{E}^{[\otimes m]}\coloneqq(\mathcal{E}^{\otimes m})^{\vee\vee}$;
\item the \textit{reflexive symmetric power} $\textup{Sym}^{[m]}\mathcal{E}$ by $\textup{Sym}^{[m]}\mathcal{E}\coloneqq(\textup{Sym}^m\mathcal{E})^{\vee\vee}$; 
\item the \textit{determinant sheaf} $\det\mathcal{E}$ by the reflexive top exterior power $\wedge^{[\textup{rank}\,\mathcal{E}]}\mathcal{E}\coloneqq (\wedge^{\textup{rank}\,\mathcal{E}}\mathcal{E})^{\vee\vee}$,  where $\textup{rank}\,\mathcal{E}$ is the rank of $\mathcal{E}$;
\item the \textit{reflexive pullback} \(f^{[*]}\mathcal{E}\) by $f^{[*]}\mathcal{E} \coloneqq (f^{*}\mathcal{E})^{\vee\vee}$ for a morphism $f \colon  Y \rightarrow X$ between normal  varieties.
\end{enumerate}
\end{definition}

\begin{definition}[Sheaf of reflexive logarithmic 1-form]
Let \(D\) be an effective reduced divisor on a normal projective variety \(X\).
Let \(U\) be a big open subset (whose complement in \(X\) is codimension \(\geq 2\)) of the pair \((X,D)\) such that \((X,D)\) is log smooth.
We denote by \(\Omega_X^{[1]}(\textup{log}\,D)\) the reflexive sheaf on \(X\) whose restriction to \(U\) is the sheaf of the usual logarithmic differential 1-form \(\Omega_U^1(\textup{log}\,D|_U)\).
The reflexive sheaf \(\Omega_X^{[1]}(\textup{log}\,D)\) is called the 
\textit{sheaf of reflexive logarithmic 1-form}. 
Its dual \(T_X(-\textup{log}\,D)\) is called the \textit{logarithmic tangent sheaf}.
\end{definition}

The sheaf of reflexive logarithmic differential 1-forms is well-behaved under the finite pullback.

\begin{lemma}[{see \cite[Lemma 2.7]{DLB22}}]\label{lem-diff-finite}
Let \(\pi\colon Y\to X\) be a finite cover between normal varieties, and let \(D\) be a reduced effective divisor on \(X\).
Suppose that \(\pi\) is quasi-\'etale over \(X\backslash\textup{Supp}\,D\) and set \(D_Y\coloneqq\gamma^*(K_X+D)-K_{Y}\).
Then \(D_Y\) is reduced and effective.
Moreover, we have an isomorphism \(\pi^{[*]}\Omega_X^{[1]}(\textup{log}\,D)\cong\Omega_{Y}^{[1]}(\textup{log}\,D_Y)\).
\end{lemma}

\begin{definition}[Toric varieties]
A normal variety $X$ of dimension $n$ is a \textit{toric variety} if $X$ contains a {\it big torus} $T=(\mathbf{k}^*)^n$ as an (affine) open dense subset such that the natural multiplication action of $T$ on itself extends to an action on the whole variety. In this case, let $B\coloneqq X\backslash T$, which is a divisor; the pair $(X,B)$ is said to be a \textit{toric pair}. 
It is known that the sheaf of reflexive logarithmic 1-form $\Omega_X^{[1]}(\textup{log}\,B)\cong {\OO}_X^{\oplus n}$ of a toric pair \((X,B)\) is free (see e.g. \cite[Remark 4.6]{MZ19}) and $K_X+B\sim 0$.
\end{definition}

\subsection{Fibrations}
Throughout this paper, we will frequently switch within the following various fibrations.

\begin{definition}\label{defn-locally-trivial}
Let \(\pi\colon (X,D)\to Y\) be a fibration between normal varieties (i.e., a proper surjective morphism with connected fibers) where \(D\) is a Weil \(\mathbb{Q}\)-divisor on \(X\). 
\begin{enumerate}
\item We say that \(\pi\) is an \textit{analytically (resp. Zariski) locally trivial fibration over \(Y\)} if for any point \(y\in Y\), there is a small (resp. Zariski) open neighborhood \(U\) of \(y\) such that \((\pi^{-1}(U),D|_{\pi^{-1}(U)})\cong U\times (F,D|_F)\) where \(F\) is a typical fiber of \(\pi\).
\item We say that an analytically locally trivial fibration \(\pi\) is a \textit{toric fibration}, if for any fiber \(F\) of \(\pi\), the restriction \((F,D|_F)\) is a toric pair. 
\item We say that a toric fibration \(\pi\)  is a  \textit{splitting toric fibration}, if for each irreducible component \(D_i\) of \(D\), the restriction \(D_i|_F\) is irreducible for a general fiber \(F\) of \(\pi\). We shall see in Theorem \ref{thm-toric-trivial} that a splitting toric fibration is always Zariski locally trivial. 
\end{enumerate}
\end{definition}

\begin{definition}[Albanese map]
Let \(X\) be a normal projective variety.
Denote by \(\Alb(X)\coloneqq\Pic^0(\Pic^0(X)_{\textup{red}})\) which is an abelian variety.
There is a canonical morphism 
$$\alb_X\colon X\to\Alb(X)$$ such that the image \(\alb_X(X)\) generates \(\Alb(X)\) and every morphism \(X\to A\) from \(X\) to an abelian variety \(A\) factors through \(\Alb(X)\) (see \cite[Remark 9.5.25]{FGA05}).
If $X$ further has only rational singularities (e.g., klt), then $\dim \Alb(X)=q(X)$.
\end{definition}

\subsection{Endomorphisms and dynamical invariants}

\begin{definition}\label{def-endo}
Let \(f\colon X\to X\) be a surjective endomorphism of a normal projective variety \(X\). 
\begin{enumerate}
\item We say that \(f\) is \textit{$q$-polarized} if $f^*H\sim qH$ for some ample Cartier divisor $H$ and integer $q>1$, or equivalently,  \(f^*|_{\textup{N}^1(X)}\) is diagonalizable with all the eigenvalues being of modulus \(q\) (see \cite[Proposition 2.9]{MZ18}).
\item We say that  $f$ is \textit{int-amplified} if $f^*L-L$ is ample for some  ample  Cartier divisor $L$, or equivalently, all the eigenvalues of \(f^*|_{\textup{N}^1(X)}\) are of modulus greater than 1 (see \cite[Theorems 1.1 and 3.3]{Men20}). 
Clearly, every polarized endomorphism is int-amplified.
\item A subset $D\subseteq X$ is said to be {\it $f^{-1}$-invariant}  if $f^{-1}(D)=D$. 
\item Denote by $\SEnd(X,D)$ the monoid of surjective endomorphisms $f$ of $X$ with $f^{-1}(D)=D$.
\end{enumerate}
\end{definition}

We refer to \cite{MZ23} for a quick survey of the recent progress on polarized and int-amplified endomorphisms.

\begin{definition}[First dynamical degree]\label{def-dyn}
Let \(f\colon X\to X\) be a surjective endomorphism of a normal projective variety \(X\) and \(H\) an ample divisor on \(X\).
The \textit{first dynamical degree} \(\delta_f\) of \(f\) is defined to be the following limit
\[
\delta_f\coloneqq\lim_{n\to\infty}((f^n)^*H\cdot H^{\dim X-1})^{1/n}\in\mathbb{R}_{\geq 1}.
\]
It is known that the limit always exists and is independent of the choice of the ample divisor  (see \cite{DS04,DS05}; cf.~\cite{Dan19}). 
It is also known that the first dynamical degree \(\delta_f\) is invariant under generically finite maps and coincides with the spectral radius of the induced linear operation \(f^*|_{\textup{NS}_{\mathbb{C}}(X)}\). 
Note that \(\delta_{f^s}=(\delta_f)^s\).
\end{definition}

\begin{definition}[Arithmetic degree]\label{def-ari}
Let \(f\colon X\to X\) be a surjective endomorphism of a normal projective variety \(X\) over \(\overline{\mathbb{Q}}\).
\begin{enumerate}
\item 
For  \(D\in\textup{N}^1(X)\) on \(X\), there is a height function \(h_D\colon X(\overline{\mathbb{Q}})\to\mathbb{R}\) associated to \(D\) which measures the complexity of \(\overline{\mathbb{Q}}\)-points. 
Such a function is uniquely determined up to a bounded function.
We refer to \cite[Part B]{HS00} (cf.~\cite[Theorem 1.1.1]{Kaw06}) for an introduction to
Weil’s height theory. 
\item Let \(h_H\geq 1\) be an absolute logarithmic Weil height function associated with an ample divisor \(H\).
Then for every \(x\in X(\overline{\mathbb{Q}})\), we define the arithmetic degree of \(f\) at \(x\) by
\[
\alpha_f(x)=\lim_{n\to\infty}h_H(f^n(x))^{1/n}\in\mathbb{R}_{\geq 1}.
\]
It is known that the limit always exists and is also independent of the choice of the ample divisor (see \cite[Proposition 12]{KS16a}, \cite[Theorem 3 (a)]{KS16b} for details). 
Note also that \(\alpha_f(x)\leq \delta_f\) and  \(\alpha_{f^s(x)}=\alpha_f(x)^s\). 
\end{enumerate}
\end{definition}

The following lemma allows us to reduce the Kawaguchi-Silverman conjecture (KSC, Conjecture \ref{conj-KSC}) to its generically finite cover of the underlying variety.
\begin{lemma}[{see \cite[Lemma 2.5]{MZ22}}]\label{lem-ksc-iff}
Consider the equivariant dynamical systems
$$\xymatrix{
f \acts X \ar@{-->}[r]^{\pi} &Y\racts g
}$$
of normal projective varieties with $\pi$ a dominant rational map.
Then the following hold.
\begin{itemize}
\item[(1)] Suppose $\pi$ is generically finite.
Then KSC holds for $f$ if and only if KSC holds for $g$.
\item[(2)] Suppose $\delta_f=\delta_g$ and KSC holds for $g$.
Then KSC holds for $f$.
\end{itemize}
\end{lemma}

\section{Transcendental argument}\label{sec-trans}
In our paper, we address the arithmetic problem that operates over the field \(\overline{\mathbb{Q}}\) of algebraic numbers. However, most advanced theories of vector bundles and fundamental groups are initially developed in the analytical setting. Although this should be well-known to experts, for the sake of completeness, we briefly explain in this section how to apply the known results from \(\mathbb{C}\) to any algebraically closed field \(\mathbf{k}\) of characteristic zero.

Let \(X\) be a projective variety defined over an algebraically closed field \(\mathbf{k}\) of characteristic zero, \(D\) a reduced effective divisor on \(X\), and \(\mathcal{E}\) is reflexive sheaf on \(X\).
Since the defining field of \(X\) is indeed finitely generated over \(\mathbb{Q}\), we may assume that \(\mathbf{k}\) is a subfield of \(\mathbb{C}\). Denote by \(X_{\mathbb{C}}\coloneqq X\times_{\textup{Spec}\,\mathbf{k}}\textup{Spec}\,\mathbb{C}\), \(D_{\mathbb{C}}\coloneqq D\times_{\textup{Spec}\,\mathbf{k}}\textup{Spec}\,\mathbb{C}\), and \(\mathcal{E}_{\mathbb{C}}\coloneqq\mathcal{E}\times_{\textup{Spec}\,\mathbf{k}}\textup{Spec}\,\mathbb{C}\). Let \(\pi\colon Y\to X_{\mathbb{C}}\) be a finite cover which is quasi-\'etale away from \(D_{\mathbb{C}}\).

By the Lefschetz principle, there is a finitely generated \(\mathbf{k}\)-algebra \(R\) such that \(\pi\), \(Y\), \(\pi^*D_{\mathbb{C}}\), \(\pi^{[*]}\mathcal{E}\) and \(X\) are all defined over \(\textup{Spec}\,R\).
In other words, there exists a finite morphism \(\pi_V\colon Y_V\to X_V\) of normal projective varieties defined over \(V\coloneqq\textup{Spec}\,R\) such that \(X_{\mathbb{C}}=X_V\times_{V}\textup{Spec}\,\mathbb{C}\), \(Y=Y_V\times_{V}\textup{Spec}\,\mathbb{C}\) and \(\pi=\pi_V\times_{V}\textup{Spec}\,\mathbb{C}\), noting that \(R\) is a subring of \(\mathbb{C}\) as \(R\) is constructed by taking the defining coefficients of \(Y\) and \(\pi\). 
So we obtain the following commutative diagram
\[
\xymatrix{
Y\ar[r]\ar[d]_{\pi}&Y_V\ar[d]^{\pi_V}&\\
X_{\mathbb{C}}\ar[r]\ar[d]_{p_\mathbb{C}}&X_V\ar[r]\ar[d]^{p_V}&X\ar[d]_{p_{\mathbf{k}}}\\
\textup{Spec}\,\mathbb{C}\ar[r]&V\ar[r]&\textup{Spec}\,\mathbf{k}
}
\]

With the notation above, we prove the following theorem.
\begin{theorem}\label{thm-trans}
Suppose that the reflexive pullback \(\pi^{[*]}\mathcal{E}\) is locally free.
There exists a finite cover \(\pi_{\mathbf{k}}\colon Y_{\mathbf{k}}\to X\) defined over \(\mathbf{k}\) which is quasi-\'etale away from \(D\) such that the reflexive pullback \(\pi_{\mathbf{k}}^{[*]}\mathcal{E}\) is also locally free and  \(q(Y_{\mathbf{k}})=q(Y)\).
\end{theorem}

\begin{proof} 
Since the natural ring extension \(R\to\mathbb{C}\) is flat, and the flatness is stable under the base change, by the descent theory, it follows that \(\pi\) is quasi-\'etale away from \(D_{\mathbb{C}}\) if and only if \(\pi_V\) is quasi-\'etale away from \(D_V\) over the generic point of \(V\).
After shrinking \(V\), we may assume that the branch divisor \(B_{\pi_V}\) of \(\pi_V\) is contained in \(D_V\). 
By the generic flatness, after shrinking \(V\), we may also assume that \(\pi_V\) is flat.  
Moreover, as \(\pi^{[*]}\mathcal{E}_{V}\) is locally free along the generic fiber of the projective morphism \(p_V\circ\pi_V\), the closed subset over which \(\pi^{[*]}\mathcal{E}_{V}\) is not locally free does not dominate \(V\).
In particular, after a further shrinking of \(V\), we may assume that \(\pi^{[*]}\mathcal{E}_{V}\) is locally free over \(V\). 
Hence, we can pick a general \(\mathbf{k}\)-rational point \(u\in  V\) such that \(\pi_u\colon Y_{V,u}\coloneqq \pi_V^{-1}p_V^{-1}(u)\to X_{V,u}\coloneqq p_V^{-1}(u)\) is a finite morphism defined over \(\mathbf{k}\) and quasi-\'etale away from \(D_V|_{p_V^{-1}(u)}\). 
Besides, by the upper semi-continuity (see~\cite[Chapter III, Theorem 12.8]{Har77}), we may assume \(q(Y_{V,u})=q(Y)\) after further shrinking of $V$, noting that \(q(Y)\) coincides with the irregularity of the generic fiber of \(p_V\circ\pi_V\). 
Since there is a natural isomorphism \((X_V,D_V)=(X,D)\times_{V}\textup{Spec}\,\mathbf{k}\),  
 we conclude our lemma by taking \(Y_{\mathbf{k}}=Y_{V,u}\) and \(\pi_{\mathbf{k}}=\pi_{u}\).
\end{proof}

\section{Positivity of dynamically invariant reflexive sheaves}\label{sec-pos}
In this section, we study the dynamical restrictions on the invariant reflexive sheaf, with Theorem \ref{thm-invariant-sheaf} as the main result. We first recall several notions of positivity.
\begin{definition}\label{def-pos}
Let \(\mathcal{E}\) be a torsion free coherent sheaf on a normal projective variety \(X\).
We say that \(\mathcal{E}\) is 
\begin{enumerate}
\item  \textit{generically globally generated} if \(\mathcal{E}\) is globally generated at a general point, i.e., there exists a generically surjective map \(\mathcal{O}_X^{\oplus n}\to \mathcal{E}\) for some \(n\); 
\item  \textit{pseudo-effective} (or \textit{weakly positive}) 
if for any positive integer \(a\in\mathbb{Z}\) and for any ample divisor \(A\), there exists a positive integer \(b\) such that \(\textup{Sym}^{[ab]}\mathcal{E}\otimes\mathcal{O}_X(bA)\) is globally generated along a general point (see \cite[Section 2]{Mat23} for other equivalent definitions; cf.~\cite[Chapter V. Definition 3.20]{Nak04}); 
\item \textit{nef} if the tautological line bundle \(\mathcal{O}_{\mathbb{P}_X(\mathcal{E})}(1)\) is nef on the Grothendieck projectivization \(\mathbb{P}_X(\mathcal{E})\coloneqq \textup{Proj}(\textup{Sym}^{\bullet}\mathcal{E})\); 
\item \textit{almost nef} if the base field is \textbf{uncountable} (e.g., \(\mathbb{C}\)) and there exist countably many proper subvarieties \(Z_i\subseteq X\) such that the sheaf \(\mathcal{E}|_C\coloneqq\mathcal{E}\otimes\mathcal{O}_C\) is nef for any curve \(C\not\subseteq\cup_iZ_i\)
(we only use this notation in the proof of Theorem \ref{thm-struc-k}); 
\item \textit{numerically flat} if both \(\mathcal{E}\) and its dual \(\mathcal{E}^{\vee}\) are nef. 
\item \textit{R-flat} if \(\mathcal{E}\) is locally free of rank $n$ and \(\nu^*\mathcal{E}\cong\mathcal{O}_{\mathbb{P}^1}^{\oplus n}\) for any morphism \(\nu\colon\mathbb{P}^1\to X\).
Clearly, a numerically flat locally free sheaf is always R-flat and the converse holds true if \(X\) is rationally connected (see \cite[Lemma 2.13]{DLB22}). 
\end{enumerate}
\end{definition}

The main theorem in this section is a natural generalization of the Weil divisor case. Indeed, when \(\mathcal{E}\) is a Weil divisor, it follows immediately from \cite[Theorem 3.3]{Men20} that \(\mathcal{E}\) is weakly numerically trivial, since all the eigenvalues of \(f^*|_{\N_{n-1}(X)}\) have modulus greater than 1.

\begin{theorem}\label{thm-invariant-sheaf}
Let \(f \colon X \to X\) be an int-amplified endomorphism of a 
normal projective variety \(X\) of dimension \(n\).
Let \(\mathcal{E}\) be a reflexive sheaf on \(X\) such that the reflexive pullback \(f^{[*]}\mathcal{E} \cong \mathcal{E}\). 
Then \(\mathcal{E}\) is pseudo-effective.
Moreover,  \(\pi^{[*]}\mathcal{E}\) is a numerically flat locally free sheaf for some quasi-\'etale cover \(\pi \colon Y \to X\) if \(X\) is further klt.
\end{theorem}

Before proving Theorem \ref{thm-invariant-sheaf}, we extend \cite[Proposition 3.1]{IMZ23} to the following proposition, which holds over any algebraically closed field of characteristic zero.

\begin{proposition}\label{prop-imz-k}
Let \(X\) be a klt projective variety of dimension \(n\) (defined over \(\mathbf{k}\)) and \(\mathcal{E}\) a pseudo-effective reflexive sheaf such that \(\det(\mathcal{E})\cdot H^{n-1}=0\) for some ample divisor $H$. 
Then there is a quasi-\'etale cover \(\pi_{\mathbf{k}}\colon Y\to X\) such that the reflexive pullback \(\pi_{\mathbf{k}}^{[*]}\mathcal{E}\) is a numerically flat locally free sheaf.
\end{proposition}

\begin{proof}
Following the same reason as in Section \ref{sec-trans}, we may assume that \(\mathbf{k}\) is a subfield of \(\mathbb{C}\).
Denote by \(X_{\mathbb{C}}\coloneqq X\times_{\textup{Spec}\,\mathbf{k}}\textup{Spec}\,\mathbb{C}\) and
\(\mathcal{E}_{\mathbb{C}}, H_{\mathbb{C}}\) similarly. 
Note that \(\det(\mathcal{E}_{\mathbb{C}})\cdot H_{\mathbb{C}}^{d-1}=0\) and \(\mathcal{E}_{\mathbb{C}}\) is pseudo-effective. 
By \cite[Proposition 3.1]{IMZ23}, we know that there is a quasi-\'etale cover \(\pi\colon Y_0\to X_{\mathbb{C}}\) such that \(\pi^{[*]}\mathcal{E}_{\mathbb{C}}\) is locally free. 
By Theorem \ref{thm-trans}, we obtain a finite quasi-\'etale cover \(\pi_{\mathbf{k}}\colon Y\to X\) (defined over \(\mathbf{k}\)) such that \(\pi_{\mathbf{k}}^{[*]}\mathcal{E}\) is also locally free.
We are left to prove that \(\pi_{\mathbf{k}}^{[*]}\mathcal{E}\) and \((\pi_{\mathbf{k}}^{[*]}\mathcal{E})^{\vee}\) are both nef.
It is equivalent to proving that the locally free sheaves \((\pi_{\mathbf{k}}^{[*]}\mathcal{E})_{\mathbb{C}}\) and \((\pi_{\mathbf{k}}^{[*]}\mathcal{E})_{\mathbb{C}}^{\vee}\) are both nef.
By noting that \((\pi_{\mathbf{k}}^{[*]}\mathcal{E})_{\mathbb{C}}\) is pseudo-effective and \(\det((\pi_{\mathbf{k}}^{[*]}\mathcal{E})_{\mathbb{C}})\cdot (\pi_{\mathbf{k}}^*H)_{\mathbb{C}}^{d-1}=0\), we can conclude by \cite[Proposition 3.1]{IMZ23} once more.
\end{proof}

\begin{proof}[Proof of Theorem \ref{thm-invariant-sheaf}]
Assume that \(f^*H-H\) is ample for some ample Cartier divisor \(H\).
Hence, there is some rational number \(r>1\) such that \(f^*H-rH\) is an ample \(\mathbb{Q}\)-Cartier divisor. 
With \(f\) replaced by a suitable power, we may assume that \(r\) is an integer. 
We may further assume that \(f^*H-rH\) is very ample with \(H\) replaced by a sufficient multiple.

\begin{claim}\label{claim-twist-psef}
The sheaf $\textup{Sym}^{[m]}\mathcal{E}\otimes\mathcal{O}_X(H)$ is pseudo-effective
for any \(m\in\mathbb{Z}_{>0}\).   
\end{claim}
\begin{proof}
As \(H\) is ample, for each reflexive sheaf \(\textup{Sym}^{[m]}\mathcal{E}\), there exists some \(t_m\in\mathbb{Z}_{>0}\) such that \(\textup{Sym}^{[m]}\mathcal{E}\otimes \mathcal{O}_X(tH)\) is globally generated whenever \(t\geq t_m\). 
Besides, for any \(l\in\mathbb{N}\), we have
\[
(f^l)^*H-r^lH=\sum_{i=1}^lr^{l-i}(f^{i-1})^*(f^*H-rH).
\]
As \(f^*H-rH\) is very ample by assumption, the above sum 
\((f^l)^*H-r^lH\) is ample and base point free for any $\ell>0$.
Fix one \(l\) such that \(r^l>t_m\).
Hence, it follows that
\begin{align*}
(f^l)^{[*]}(\textup{Sym}^{[m]}\mathcal{E}\otimes\mathcal{O}_X(H))&\cong\textup{Sym}^{[m]}\mathcal{E}\otimes\mathcal{O}_X((f^{l})^*H)\\
&\cong\textup{Sym}^{[m]}\mathcal{E}\otimes\mathcal{O}_X(r^lH)\otimes\mathcal{O}_X((f^l)^*H-r^lH)
\end{align*}
is globally generated and thus pseudo-effective.
Note that the above isomorphism holds over a big open subset (whose complement is of codimension \(\geq 2\) in \(X\)) and hence the isomorphism holds by the unique extension of reflexive sheaves. 
By \cite[Proposition 2.6]{Mat23} (the same proof of which works for the base field \(\mathbf{k}\)), we obtain that \(\textup{Sym}^{[m]}\mathcal{E}\otimes\mathcal{O}_X(H)\) is pseudo-effective for any \(m\). 
Our claim is thus proved.
\end{proof}

Let us come back to the proof of Theorem \ref{thm-invariant-sheaf}. 
Let \(A\) be any ample Cartier divisor on \(X\) and let \(a\) be any positive integer. 
Take a sufficiently large integer \(u\) such that \(uA-H\) is ample. 
As \(\textup{Sym}^{[a\cdot u]}\mathcal{E}\otimes\mathcal{O}_X(H)\) is pseudo-effective, 
 for the given ample divisor \(uA-H\) and the positive integer 1,  
there exists some \(b\in \mathbb{Z}_{>0}\) such that
\[
\textup{Sym}^{[b\cdot 1]}(\textup{Sym}^{[au]}\mathcal{E}\otimes\mathcal{O}_X(H))\otimes\mathcal{O}_X(b(uA-H))
\]
is generically globally generated. 
In particular, 
\[\textup{Sym}^{[b]}(\textup{Sym}^{[au]}\mathcal{E})\otimes\mathcal{O}_X(buA)\]
is generically globally generated.
Now that we have the following generically  surjective map
\[
\textup{Sym}^{[b]}\textup{Sym}^{[au]}\mathcal{E}\otimes\mathcal{O}_X(buA)\to \textup{Sym}^{[abu]}\mathcal{E}\otimes \mathcal{O}_X(buA),
\]
we have found a positive integer \(bu\) such that \( \textup{Sym}^{[abu]}\mathcal{E}\otimes \mathcal{O}_X(buA)\) is globally generated along a general point. 
This finishes the first half of the theorem (see Definition \ref{def-pos}). 

For the second half, we first claim the following.
\begin{claim}\label{claim-determinant}
We have \(\det(\mathcal{E})\cdot H^{n-1}=0\) for the above given ample Cartier divisor \(H\).
\end{claim}
\begin{proof}[Proof of Claim \ref{claim-determinant}]
Note that \(X\) is normal and $\mathcal{E}$ is torsion free. 
Let \(U\) be a Zariski open subset of $X$ with $\Codim X\backslash U\ge 2$ such that \(\mathcal{E}\) and \(\det(\mathcal{E})\) are both locally free on $U$.
Therefore, we have $$(f|_{f^{-1}(U)})^*(\det(\mathcal{E}|_U))=\det((f|_{f^{-1}(U)})^*\mathcal{E}|_U)$$ by the functoriality. 
By the unique reflexive extension, we have \(f^{[*]}(\det(\mathcal{E}))\cong\det(\mathcal{E})\). 
Let \(\det(\mathcal{E})=\mathcal{O}_X(D)\) be the divisorial sheaf of a Weil divisor \(D\).
Then we have \(f^*D\sim D\) and thus 
we conclude our claim by applying \cite[Lemma 3.8]{Men20}.
\end{proof}
Let us come back to the proof of Theorem \ref{thm-invariant-sheaf}.
By Claim \ref{claim-determinant}, together with the pseudo-effectiveness of  \(\mathcal{E}\), it follows from Proposition \ref{prop-imz-k} that there is a quasi-\'etale cover \(\pi\colon Y\to X\) (defined over \(\mathbf{k}\)) such that \(\pi^{[*]}\mathcal{E}\) is a numerically flat locally free sheaf.
\end{proof}

We close this section with the following theorem, which slightly generalizes \cite[Theorem 6.1]{DLB22} (cf.~\cite[Theorem 1.4]{Iwa22}) from the field of complex numbers to any algebraically closed field of characteristic zero (cf.~Remark \ref{rmk-diff-mrc}).

\begin{theorem}\label{thm-struc-k}
Let $X$ be a klt projective variety over an algebraically closed field $\mathbf{k}$.
Let $D$ be a reduced divisor such that \(\Omega_X^{[1]}(\textup{log}\,D)\) is numerically flat and locally free. 
Then \(\widetilde{q}(X)<\infty\) holds (see~Definition \ref{defn-irreg}). 
Suppose further \((X,D)\) is a projective lc pair and \(q(X)=\widetilde{q}(X)\).
Then \(\alb_X\colon(X,D)\to A\) is a toric fibration (see~Definition \ref{defn-locally-trivial}). 
\end{theorem}

\begin{proof}
With the same reason as in Section \ref{sec-trans} and following the notations therein, we may assume that \(\mathbf{k}\) is a subfield of \(\mathbb{C}\). 
Note that there is a generically surjective map
\(T_{X_{\mathbb{C}}}(-\textup{log}\,D_{\mathbb{C}})\to T_{X_{\mathbb{C}}}\) 
and the logarithmic tangent bundle \(T_{X_{\mathbb{C}}}(-\textup{log}\,D_{\mathbb{C}})\) is numerically flat and thus almost nef. 
By \cite[Lemma 2.6 (4)]{IMZ23}, we know that the tangent sheaf \(T_{X_{\mathbb{C}}}\) is also almost nef. 
Then it follows from Theorem \ref{thm-trans} and \cite[Lemma 6.1]{IMZ23} that
\(\widetilde{q}(X)=\widetilde{q}(X_{\mathbb{C}})<\infty\). 
Also, from \cite[Theorem 6.1]{DLB22}, we obtain a toric fibration
\(u\colon (X_{\mathbb{C}},D_{\mathbb{C}})\to Q\) onto a smooth projective variety \(Q\) with \(c_1(Q)=0\). 
Moreover, since \(u\) is surjective and \(Q\) is smooth, by dualizing the sheaf sequence
\(0\to u^*\Omega_Q\to \Omega_{X_{\mathbb{C}}}\),  
we obtain another generically surjective map
\(T_{X_{\mathbb{C}}}\to u^*T_Q\). 
This implies that \(T_Q\) is also almost nef (see~\cite[Lemma 2.6 (2)]{IMZ23}). 
Since \(K_Q\equiv 0\), by applying \cite[Corollary 3.2]{IMZ23}, \(Q\) is an \'etale quotient of an abelian variety. 

We claim that \(u\) is indeed the Albanese morphism of \(X_{\mathbb{C}}\) and \(Q\) is an abelian variety. 
Indeed, as the Albanese map \(\alb_{X_{\mathbb{C}}}\colon X_{\mathbb{C}}\to A\) contracts all of the fibers of \(u\), it follows from the rigidity lemma \cite[Lemma 1.15]{Deb01} that \(\alb_{X_{\mathbb{C}}}\) factors through \(u\). 
On the other hand, by \cite[Lemma 6.1]{IMZ23}, \(\alb_{X_{\mathbb{C}}}\) is surjective and thus \(\dim Q\geq \dim A=q(X)=\widetilde{q}(X)\geq \dim Q\).
This implies that \(Q\to A\) is a finite \'etale morphism and in particular, \(Q\) is an abelian variety.  
By the universality of the Albanese morphism, \(u\) also factors through \(\alb_{X_{\mathbb{C}}}\), which concludes the proof of our claim.

Finally, note that the Albanese morphism is independent of the base field.
So the theorem is proved.
\end{proof}

\begin{remark}\label{rmk-diff-mrc}
In the proof of Theorem \ref{thm-struc-k}, we are unable to show that the initial toric fibration \(u\) in \cite[Theorem 6.1]{DLB22} is defined over \(\mathbf{k}\). However, after a base change, and in light of Theorem \ref{thm-trans}, we are in a situation of Albanese morphisms, which are independent of the choice of the base field.
\end{remark}

\section{Zariski local triviality and splitting of toric fibration}\label{sec-fib}
In this section, we study the toric fibration and its equivariant lifting.
We refer the reader to Definition \ref{defn-locally-trivial} for the relevant notion. 
We thank J. Xie and D.-Q. Zhang for the communication of this section; the proofs of Lemma \ref{lem-basechange} and Theorem \ref{thm-toric-trivial} are essentially borrowed from their ongoing collaborative work with the first author \cite{MXZ24}.

We begin with the following lemma, which demonstrates that after an equivariant base change, a horizontal subvariety of a fibration will be of splitting type.

\begin{lemma}\label{lem-basechange}
Let \(\pi\colon X\to Y\) be a fibration of normal projective varieties with connected fibers. 
Let $D$ be a closed subvariety with each irreducible component dominating $Y$. 
Then we have the following commutative diagram
\[
\xymatrix{
\widehat{X}\ar[r]^{\widehat{\pi}}\ar[d]_{p_X}&\widehat{Y}\ar[d]^{p_Y}\\
X\ar[r]^{\pi}&Y
}
\]
where \(p_Y\colon\widehat{Y}\to Y\) is a finite surjective from a normal projective variety \(\widehat{Y}\), \(\widehat{X}\) is the main component of the normalization of the fiber product \(X\times_Y\widehat{Y}\), and \(\widehat{\pi}|_{\widehat{D_i}}\colon\widehat{D_i}\to\widehat{Y}\) has irreducible general fibers for any irreducible component \(\widehat{D_i}\) of \(p_X^{-1}(D)\). 

Moreover, for any  surjective endomorphisms \(f\colon X\to X\) and \(g\colon Y\to Y\) such that \(\pi\circ f=g\circ \pi\) and \(f(D)=D\), we have the following equivariant dynamical systems:
$$\xymatrix{
\save[]+<-1.4pc,0.15pc>*{\widehat{f} \acts} \restore \widehat{X} \ar[r]^{\widehat{\pi}}\ar[d]_{p_X} &\widehat{Y}\ar[d]^{p_Y}\save[]+<1.4pc,0.15pc>*{\racts\widehat{g}} \restore \\
\save[]+<-1.4pc,0.15pc>*{f \acts} \restore  X \ar[r]^{\pi} &Y\save[]+<1.4pc,0.15pc>*{\racts g} \restore 
}$$
where \(\widehat{f}\) and \(\widehat{g}\) are the induced surjective endomorphisms.
\end{lemma}

\begin{proof}
It suffices for us to consider the case when $D$ is irreducible.
Let $n_D\colon\overline{D}\to D$ be the normalization.
Let $\sigma\colon\overline{D}\to \widehat{Y}$ and $p_Y\colon\widehat{Y}\to Y$ be the Stein factorization of $\pi|_D\circ n_D$.
If $\deg (p_Y)=1$, then $\pi|_D\colon D\to Y$ has irreducible general fibers and we do no base change. 
Suppose $\deg (p_Y)>1$. 
Then the natural embedding $\widehat{Y}\hookrightarrow \widehat{Y}\times_Y \widehat{Y}$ implies that $\widehat{Y}\times_Y \widehat{Y}$ splits into at least two irreducible components, noting that the surjective morphism \(\sigma\) induces a surjective morphism \(\overline{D}\times_Y\widehat{Y}\to \widehat{Y}\). 
This shows that $D\times_Y \widehat{Y}$ contains more irreducible components (dominating $Y$) than $D$. 
Note that $p_X^{-1}(D)=D\times_Y \widehat{Y}$ as sets.
So we can repeat the above operation for each irreducible component $\widehat{D}_i$ (dominating $Y$) of $p_X^{-1}(D)$ until each $\widehat{\pi}|_{\widehat{D}_i}\colon\widehat{D}_i\to \widehat{Y}$ has irreducible general fibers. 
The first half of the lemma is thus completed. 

For the second half, we note that there is a surjective endomorphism $h\colon \overline{D}\to \overline{D}$ such that $f|_D\circ n_D=n_D\circ h$.
Then it follows from  \cite[Lemma 5.2]{CMZ20} that there is a surjective endomorphism $\widehat{g}\colon \widehat{Y}\to \widehat{Y}$ commuting with \(h\) and \(g\).
So we obtain our lemma.
\end{proof}

\begin{lemma}\label{lem-toric-splitting-bc}
With the same assumption as in Lemma \ref{lem-basechange}, suppose further that \(f\) is int-amplified and \(\pi\colon (X,D)\to Y\) is a toric firbation over an abelian variety $Y$.
Then we can further require \(\widehat{\pi}\colon (\widehat{X}, p_X^{-1}(D))\to \widehat{Y}\) to be a \textbf{splitting toric} fibration over an abelian variety \(\widehat{Y}\). 
\end{lemma}

\begin{proof}
As \(f\) is int-amplified, it follows from \cite[Lemmas 3.5 and 3.4]{Men20} that \(\widehat{f}\) and thus \(\widehat{g}\) are both int-amplified. 
On the one hand, by \cite[Theorem 1.5]{Men20}, \(-K_{\widehat{Y}}\) is weakly numerically equivalent to some effective \(\mathbb{Q}\)-Weil divisor.
Since \(K_Y\sim 0\) and by the ramification divisor formula, \(K_{\widehat{Y}}\) is an effective divisor. 
So \(K_Y\equiv 0\) and thus \(p_Y\) is \'etale by the purity of branch locus.
Therefore, \(\widehat{Y}\), as an \'etale cover over an abelian variety, is also an abelian variety. 

Note that $\pi$ is analytically locally trivial.
Then $X\times_Y \widehat{Y}\to \widehat{Y}$ is analytically locally trivial.
In particular, we do not need to take the main component and the normalization, i.e., $\widehat{X}\cong X\times_Y \widehat{Y}$.
It is then easy to see that \(\widehat{\pi}\colon (\widehat{X}, p_X^{-1}(D))\to \widehat{Y}\) is a \textbf{splitting toric} fibration. 
\end{proof}

The following theorem plays a crucial role in the proof of Theorem \ref{thm-equi-lifting}, particularly in ensuring the reducibility of $\widetilde{D}$ and the $\mathbb{Q}$-factoriality necessary to proceed with the further Minimal Model Program (MMP).

\begin{theorem}\label{thm-toric-trivial}
Let $\pi\colon (X,D)\to Y$ be a splitting toric fibration over a normal variety \(Y\). 
Then \((X,D)\) is a Zariski locally trivial pair over \(Y\).
Moreover, there exists a birational morphism \(\sigma\colon(\widetilde{X},\widetilde{D})\to (X,D)\) such that
\begin{enumerate}
\item \(K_{\widetilde{X}}+\widetilde{D}=\sigma^*(K_X+D)\); and 
\item \(\widetilde{\pi}\colon(\widetilde{X},\widetilde{D})\to Y\) is  a \textbf{splitting smooth} toric fibration over \(Y\), where \(\widetilde{\pi}=\pi\circ\sigma\).  
\end{enumerate}
\end{theorem}

\begin{proof}
Let us prove by induction the first half of the statement. 
Write \(D=\sum\limits_{i=1}^mD_i\) into the sum of irreducible components. 
By assumption, \(D|_F\) also has \(m\) irreducible components for any fiber \(F\) of $\pi$. 
Denote by $D_I\coloneqq\bigcap_{i\in I} D_i$ where $I\subseteq \{1,\cdots, m\}$. 
Fix any point \(y\in Y\) and denote by \(X_y\cong F\) the typical fiber.
Since our result is local, we can freely replace \(Y\) by a Zariski open neighborhood of \(y\). 
Note that for each \(i\), the restriction $(D_i, \sum\limits_{j\neq i} D_j\cap D_i)$ is also a splitting toric fibration over \(Y\) (see \cite[Proposition 3.2.7]{CLS11}). 
By induction on $\dim X -\dim Y$ and after shrinking $Y$, we may assume that $D_i\cong (D_i\cap X_y) \times Y$, i.e., \(D_i\) is Zariski locally trivial over \(Y\). 
Note that every toric blowup of the pair $(X_y, D|_{X_y})$ is a blowup along some $D_I\cap F$. 
Let $\textup{Bl}_{D_I}\colon\widetilde{X}\to X$ be the blowup along $D_I$, and $\widetilde{D}=\textup{Bl}_{D_I}^{*}(D)$ the pullback.
Note that $(\widetilde{X}_y, \widetilde{D}|_{\widetilde{X}_y})$ is  still a toric pair. 
Since $D_I\cong \bigcap\limits_{i\in I} (D_i\cap F) \times Y$ by induction, we see that $(\widetilde{X}, \widetilde{D})\cong (\widetilde{X}_y, \widetilde{D}\cap \widetilde{X}_y)\times Y$ if and only if $(X,D)\cong (X_y, D\cap \widetilde{X}_y)\times Y$.
Note also that \((\widetilde{X}, \widetilde{D})\) is a splitting toric fibration over \(Y\). 
Hence, after finitely many steps of replacing $X$ by its blowup along some $D_I$, we may assume there is a birational toric morphism $\sigma_y\colon (X_y, D|_{X_y})\to (\mathbb{P}^n, \sigma_y(D\cap X_y))$. 
Let \(A\coloneqq\sum\limits_{i=1}^ma_iD_i\) be a non-negative linear combination of \(D_i\) such that 
\[A_y\coloneqq A|_{X_y}=\sum\limits_{i=1}^m a_i (D_i\cap X_y)=\sigma_y^*{\sigma_y}_*(\sum_{i=1}^m D_i)\]
is a nef and big Cartier divisor. 
Since $X_y$ is of Fano type, by inversion of adjunction (see \cite[Theorem 5.50]{KM98}), there exists some effective \(\mathbb{Q}\)-Weil divisor $M\coloneqq\sum\limits_{i=1}^m b_i D_i$ such that the pair $(X,M)$ is klt and $-(K_X+M)|_{X_y}$ is ample. 
Since $(A-(K_X+M))|_{X_y}$ is ample on $X_y$, it follows from \cite[Proposition 1.41]{KM98} that $A-(K_X+M)$ is $\pi$-ample after a further shrinking of $Y$. 

We claim that $A$ is $\pi$-nef.
Suppose to the contrary that $A\cdot C<0$ for some curve \(C\) with $\pi(C)$ being a point. 
Then we may assume $C\subseteq D_1$.
Since $D_1=(D_1\cap X_y)\times Y$ by induction, we can find a subvariety \(C\times Y\) inside \(D_1\) and thus  we can find another curve $C'\coloneqq (C\times Y)\cap X_y \subseteq D_1\cap X_y$ such that $C\equiv C'$ in $D_1$. 
But then, $0>A\cdot C=A|_{D_1}\cdot C=A|_{D_1}\cdot C'=A\cdot C'$, a contradiction to the nefness of $A|_{X_y}$. 
So the claim is proved.

By the relative base-point-free theorem \cite[Theorem 3.24]{KM98}, $tA$ is $\pi$-free for $t\gg 1$. 
Then we obtain the following commutative diagram over \(Y\):
$$\xymatrix{
(X,D)\ar[rd]_{\pi}\ar[rr]^{\sigma} && (Z, \sigma(D))\ar[ld]^{\tau}\\
&Y
}$$
where $\sigma\colon X\to Z$ is the Iitaka fibration of $A$ over $Y$.
We may assume that $\sigma(D_1)$ is still a divisor on $Z$. 
By construction, we have $\sigma|_{X_y}=\sigma_y$ and $h^0(Z_y\cong \mathbb{P}^n, \sigma(D_1)|_{Z_y})=n+1$, noting that \(\sigma_y\) is indeed the Iitaka fibration of \(A_y\).
By the generic flatness and the existence of a smooth fiber \(Z_y\), after a further shrinking of \(Y\), 
we may assume that \(\tau\) is smooth,  \(\tau_*\mathcal{O}_Z(\sigma(D_1))\) is invertible and 
 $\tau_*\mathcal{O}_Z(\sigma(D_1))$ is a rank $n+1$ locally free sheaf over $Y$ (see \cite[Chapter III, Corollary 12.9]{Har77}). 
So there is a natural morphism $Z\to \mathbb{P}_{Y}(\tau_*\mathcal{O}_Z(\sigma(D_1)))$ induced by the surjection \(\tau^*\tau_*\mathcal{O}_Z(\sigma(D_1))\to \mathcal{O}_Z(\sigma(D_1))\). 
With \(Y\) further replaced by an open neighborhood of \(y\in Y\), $Z\to \mathbb{P}_{Y}(\tau_*\mathcal{O}_Z(\sigma(D_1)))$ is an isomorphism over $Y$. 
Let $\Delta=\sigma(D)$.
Therefore, after a further shrinking of $Y$, we may assume $(Z, \Delta)\cong (Z_y, \Delta\cap Z_y)\times Y$.

Now we have the following commutative diagram:
$$\xymatrix{
&(\widetilde{Z},\widetilde{\Delta})\ar@{-->}[ld]_{p}\ar[rd]^{q}\\
(X,D)\ar[rd]_{\pi}&& (Z, \Delta)\ar[ld]^{\tau}\ar@{-->}[ll]^{\sigma^{-1}} \\
&Y
}$$
where $q=q_y\times \id_Y$ with $q_y$ being a composite of toric blowups of $(Z_y, \Delta\cap Z_y)$ such that $p|_{\widetilde{Z}_y}\colon (\widetilde{Z}_y, \widetilde{\Delta}\cap \widetilde{Z}_y)\to (X_y, D\cap X_y) $ is a well-defined toric morphism.
Since the indeterminacy locus of $p$ is closed and $\pi$ is projective, we may assume $p$ is well-defined after shrinking $Y$.
Let $C$ be a curve in $\widetilde{Z}_y$.
Since $\widetilde{Z}\cong \widetilde{Z}_y\times Y$, we see that $C\times \{y\}$ is contracted by $p$ if and only if $C\times \{y'\}$ is contracted by $p$ for any $y'\in Y$.
In particular, $X\cong X_y\times Y$ by the rigidity lemma (see~\cite[Lemma 1.15]{Deb01}).
The first half of our lemma is thus proved.

For the second half, we pick a typical fiber \(F\) and we note that the singularities of the pair \((F,D|_F)\) only appears in the form \(D_I|_F\) where \(I\subseteq\{1,\cdots,m\}\). 
Hence, we can resolve the singularities of a single fiber via the global blowups of \(D_I\). 
We claim that for each step, say \(\sigma_I\colon W\coloneqq\textup{Bl}_{D_I}X\to X\) with the exceptional divisor \(E\), the induced pair \((W,D_W\coloneqq\sigma_I^{-1}(D))\) is also a splitting toric fibration over \(Y\).
Indeed, suppose that \(\sharp I=s\), i.e., \(D_I\) is a complete intersection of \(s\) irreducible components of \(D\), which is of codimension \(s\).
Then we have 
\[
K_W+D_W=\sigma_I^*K_X+(s-1)E+\sigma^*D-(s-1)E=\sigma_I^*(K_X+D).
\]
This implies that each irreducible component of \(D_W\) restricting to each fiber is still irreducible. 
Together with the local triviality of \(D_I\to Y\) and \(X\to Y\), this implies that \((W,D_W)\) is still a splitting toric fibration over \(Y\). 
Therefore, after several blowups along \(D_I\) (which is locally trivial over \(Y\)),  
we obtain a birational morphism \(\sigma\colon (\widetilde{X},\widetilde{D})\to (X,D)\) which is a splitting toric fibration over \(Y\) such that \((\pi^{-1}(F),\pi^{-1}(D|_F))\) is a log smooth pair. 
Moreover, \(\sigma^*(K_X+D)=K_{\widetilde{X}}+\widetilde{D}\) by induction and thus our proposition is proved. 
\end{proof}

\section{Proof of Theorem \ref{thm-tir}}\label{sec-pf-thm-tir}
In this section, we prove Theorem \ref{thm-tir}. 
We first prepare the following lemma. 
\begin{lemma}\label{lem-lifting-fg}
    Let $S=\{1,\cdots, n\}$. 
    Let $\sigma, \tau\colon S\to S$ be self-maps.
    Then after replacing $\sigma$ and $\tau$ by a common iteration, there exists $i\in S$ such that 
    \begin{enumerate}
        \item $\sigma(i)=i$,
        \item $\tau(j)=j$ for $j=\tau(i)$, and
        \item $(\sigma\circ (\tau\circ\sigma)^t)(j)=i$ for some $t\ge 0$.
    \end{enumerate}
\end{lemma}

\begin{proof}
We first claim that there exists $s>0$ such that $\sigma^{2s}=\sigma^s$ and $\tau^{2s}=\tau^s$. 
It is sufficient for us to only consider \(\sigma\). 
For any \(1\leq i\leq n\), as \(S\) is a finite set, there exist positive integers \(u_i, v_i\) such that \(\sigma^{u_i+v_i}(i)=\sigma^{u_i}(i)\). 
Let \(s=\prod_{i=1}^n u_iv_i\). 
Then it follows that \(\sigma^{s+v_i}(i)=\sigma^s(i)\) and hence \(\sigma^{2s}(i)=\sigma^s(i)\) for any \(i\).
Our claim is proved.
So after a common iteration, we may assume $\sigma^2=\sigma$ and $\tau^2=\tau$. 
   Consider the sequence $a_m$ where $a_1=1$, $a_m=\sigma(a_{m-1})$ if $m$ is even, and $a_m=\tau(a_{m-1})$ if $m$ is odd.
   We have $\sigma(a_m)=a_m$ (resp.~$\tau(a_m)=a_m$) if $m$ is even (resp.~odd).
   Note that $S$ is finite.
   Then for some even numbers $m, k\ge 2$, we have $a_m=a_{m+k}$.
   Let $i=a_m$, $j=\tau(i)=a_{m+1}$, and $t=k/2-1$.
   Note that $(\sigma\circ (\tau\circ\sigma)^t)(j)=i$.
   So the lemma is proved.
\end{proof}

The following lemma is well-known (cf.~\cite[Lemma 3.4]{DLB22}), noting that, up to conjugacy, there are only finitely many subgroups of the topological fundamental group \(\pi_1(X_{\textup{reg}})\) of a given index (see also \cite[arXiv version, Proposition 3.13]{GKP16}).

\begin{lemma}\label{lem-finite-index}
Let \(X\) be a normal variety and let \(d\) be a positive integer.
Then there are only finitely many quasi-\'etale covers of \(X\) of degree \(d\) up to isomorphisms over \(X\).
\end{lemma}

\begin{proof}[Proof of Theorem \ref{thm-tir}]
Let \(\mathcal{I}\in \SEnd(X,D)\) be the int-amplified endomorphism.
By the ramification divisor formula, the assumption \(K_X+D\equiv 0\) implies that 
$$K_X+D=\mathcal{I}^*(K_X+D)$$
and hence $\mathcal{I}$ is quasi-\'etale away from $D$.
So we further have 
$$\mathcal{I}^{[*]}\Omega_{X}^{[1]}(\textup{log}\,D)=\Omega_{X}^{[1]}(\textup{log}\,D)$$
by Lemma \ref{lem-diff-finite}. 
By Theorem \ref{thm-invariant-sheaf}, there exists a quasi-\'etale cover \(\pi\colon \widehat{X}\to X\) such that  \(\Omega_{\widehat{X}}^{[1]}(\textup{log}\,\pi^*D)=\pi^{[*]}\Omega_{X}^{[1]}(\textup{log}\,D)\) 
is numerically flat and locally free.
By \cite[Theorem 1.4]{BH14}, the pair \((X,D)\) has only lc singularities.
Then \(\widetilde{q}(X)<\infty\) by Theorem \ref{thm-struc-k}.
So we may assume \(q(\widehat{X})=\widetilde{q}(X)\) by taking a sufficient quasi-\'etale cover.

Consider the following sets of (irreducible) finite covers:
$$\mathcal{S}_d=
\left\{(\varphi, V)\,\left|
\begin{array}{ll}  
\varphi\colon V\to X \textup{ is quasi-\'etale with } \deg(\varphi)=d, \\  
\Omega_{V}^{[1]}(\textup{log}\,\varphi^*D) \textup{ is numerically flat and locally free, and} \\  
q(V)=\widetilde{q}(X).
\end{array}  
  \right.\right\}\Bigg/\cong
$$
where $(\varphi_1, V_1)\cong (\varphi_2, V_2)$ if $\varphi_1=\varphi_2\circ\psi$ for some automorphism $\psi\colon V_1\to V_2$.
By Lemma \ref{lem-finite-index}, \(\mathcal{S}_d\) is finite. 
Note that $(\pi,\widehat{X})\in \mathcal{S}_{\deg (\pi)}$.
So we can find a minimal positive integer \(m\) such that \(\mathcal{S}_m\neq\emptyset\). 
Write 
$$\mathcal{S}_m=\{(\varphi_1,V_1),\cdots,(\varphi_n,V_n)\}.$$

\begin{claim}\label{claim-lift}
There exists \((\varphi_i, V_i)\in\mathcal{S}_m\) such that
\(f^s\) lifts to an element in \(\SEnd(V_i,\varphi_i^*D)\) for some $s>0$ and 
\(\SEnd(V_i,\varphi_i^*D)\) contains an int-amplified endomorphism.
\end{claim}

\begin{proof}
Let $W$ be the normalization of the following fiber product
\[
\xymatrix{
W\ar[r]\ar[d]_{\varphi}\ar[r]^{p}&V_i\ar[d]^{\varphi_i}\\
X\ar[r]_{f}&X
}
\]
where $\varphi$ is quasi-\'etale and $p$ is quasi-\'etale away from $\varphi_i^{-1}(D)$.
Let $W_0$ be an irreducible component of $W$.
Note that $\varphi|_{W_0}$ is quasi-\'etale.
So $\widetilde{q}(X)\ge q(W_0)\ge q(V_i)=\widetilde{q}(X)$ implies that $q(W_0)=\widetilde{q}(X)$.
By Lemma \ref{lem-diff-finite}, 
$$\Omega_{W_0}^{[1]}(\textup{log}\,(\varphi|_{W_0})^*D)=\Omega_{W_0}^{[1]}(\textup{log}\,(p|_{W_0})^{-1}(\varphi_i^{-1}(D)))$$ is  also numerically flat and locally free. 
Note that $\deg \varphi|_{W_0}\le \deg \varphi_i$.
By the minimality of $m$, we have that $W=W_0$ is irreducible and $(\varphi, W)\in \mathcal{S}_m$.
Consequently, \(f\) and \(\mathcal{I}\) induces, via (normalization) of the base change, self-maps \(\sigma\) (resp. \(\tau\)) on the finite set \(\mathcal{S}_m\).

We use the index $i$ to represent $(\varphi_i, V_i)$.
By Lemma \ref{lem-lifting-fg}, there exist some $i$ and some $s, t>0$ such that $\sigma^s(i)=i$, $\tau^s(j)=j$ for $j=\tau^s(i)$, and $(\sigma^s\circ(\tau^s\circ \sigma^s)^t)(j)=i$.
Then $f^s$ lifts to some element in $\SEnd(V_i, \varphi_i^*D)$.
Let $g=f^s\circ (\mathcal{I}^s\circ f^s)^t$ and $\mathcal{I}'=g\circ \mathcal{I}^k\in \SEnd(X,D)$ with $k\gg 1$.
By \cite[Theorem 1.4]{Men20}, $\mathcal{I}'$ is int-amplified.
Note that $\mathcal{I}'$ lifts to an element in $\SEnd(V_i, \varphi_i^*D)$ which is int-amplified by \cite[Lemma 3.5]{Men20}.
\end{proof}

We take $(\pi, \widehat{X})=(\varphi_i, V_i)$ as in the above claim.
Let $\widehat{f}\in \SEnd(\widehat{X}, \pi^*D)$ be the lifting of $f^s$ and $\widehat{\mathcal{I}}\in \SEnd(\widehat{X}, \pi^*D)$ the int-amplified endomorphism.
Since $\pi$ is quasi-\'etale, $\widehat{X}$ is klt and $(\widehat{X}, \pi^*D)$ is log canonical.
By Theorem \ref{thm-struc-k}, the Albanese morphism 
$$\alb_{\widehat{X}}\colon (\widehat{X},\pi^*D)\to A$$ 
is a toric fibration.

We are done by further taking a quasi-\'etale cover as in Lemma \ref{lem-toric-splitting-bc}. 
\end{proof}

At the end of this section, we propose the following question, which extends \cite[Question 1.2]{MZg23}. 
In view of Theorem \ref{thm-tir}, the question has a positive answer when \(f \colon X \to X\) has totally invariant ramifications. Moreover, Yoshikawa proves in \cite[Theorem 1.3]{Yos21} that, up to replacing \(X\) by an \(f\)-equivariant quasi-\'etale cover, the general fiber of the Albanese morphism is of Fano type, which partially answers this question.

\begin{question}\label{ques}
    Let \(X\) be a $\mathbb{Q}$-factorial klt projective variety admitting an int-amplified endomorphism \(f\).
    Then up to replacing \(X\) by a quasi-\'etale cover, the Albanese map is a toric fibration onto an abelian variety.
\end{question}

\section{Equivariant modification of toric fibration, Proof of Theorem \ref{thm-equi-lifting}}\label{sec-pf-thm-equi} 
The whole section is devoted to the equivariancy of the toric fibration and its modification with the main result Theorem \ref{thm-equi-lifting-general}.
We recall some basic facts on toric morphisms. 
\begin{definition}
Let $\Delta$ be a fan in a lattice $\mathbf{N}_{\mathbb{R}}$. 
Denote by $T_{\mathbf{N}}(\Delta)$ the induced toric variety with the big torus $T_{\mathbf{N}}$ and $e_{T_{\mathbf{N}}}$ the identity element.
Note that the support $|\Delta|$ coincides with $\mathbf{N}_{\mathbb{R}}$ when $T_{\mathbf{N}}(\Delta)$ is projective. 
\begin{enumerate}
\item A morphism $f\colon T_{\mathbf{N}_1}(\Delta_1)\to T_{\mathbf{N}_2}(\Delta_2)$ is said to be {\it toric} if it comes from a lattice homomorphism that is compatible with fans, i.e., $f(T_{\mathbf{N}_1})\subseteq T_{\mathbf{N}_2}$ and $f|_{T_{\mathbf{N}_1}}$ is a group homomorphism.  
\item Any toric morphism is uniquely determined by some group homomorphism $\phi_f\colon \mathbf{N}_1\to \mathbf{N}_2$ which is compatible with the fans, i.e., for any cone $\sigma_1\in \Delta_1$, there exists a cone $\sigma_2\in \Delta_2$ such that $\phi_{f,\mathbb{R}}(\sigma_1)\subseteq \sigma_2$. 
\item In general, if $f(T_{\mathbf{N}_1})\subseteq T_{\mathbf{N}_2}$, i.e., \(f\) sends the big torus of \(T_{\mathbf{N}_1}(\Delta_1)\) to that of \(T_{\mathbf{N}_2}(\Delta_2)\), then $f=\alpha\cdot g$ where $g$ is a toric morphism and $\alpha=f(e_{T_{\mathbf{N}_1}})\in T_{\mathbf{N}_2}$ is a multiplication (see \cite[Section 3.3]{CLS11}, \cite[Section 2]{Nak21}). 
\end{enumerate}
\end{definition}

\begin{theorem}\label{thm-equi-lifting-general}
Let \(\pi\colon (X,D)\to Y\) be a toric  fibration over a normal projective variety \(Y\). 
Then there is a generically finite surjective morphism \(\sigma\colon (\widetilde{X},\widetilde{D}\coloneqq\sigma^{-1}(D))\to (X,D)\) such that the following hold.
\begin{enumerate}
\item Let \(\widetilde{\pi}\colon (\widetilde{X},\widetilde{D})\to \widetilde{Y}\) be the Stein factorization of the composite map \((\widetilde{X},\widetilde{D})\to (X,D)\to Y\).
Then \(\widetilde{\pi}\colon (\widetilde{X},\widetilde{D})\to \widetilde{Y}\) is a splitting smooth toric  fibration.
\item Let $f\in \SEnd(X,D)$ such that $\pi$ is $f$-equivariant and \(f^*D=qD\) for some positive integer $q$.
Then after iteration, $f$ lifts to $\widetilde{f}\in \SEnd(\widetilde{X},\widetilde{D})$. 
\end{enumerate}
\end{theorem}

Before proving Theorem \ref{thm-equi-lifting-general}, we make some preparations. We begin with the following lemma, which is well-known to experts.

\begin{lemma}\label{lem-toric-ext1}
Let $X=T_{\mathbf{N}}(\Delta)$ be a toric variety of dimension $n$.
Let $f\colon T_{\mathbf{N}}\to T_{\mathbf{N}}$ be an endomorphism via $f(t_1,\cdots, t_n)=(t_1^q,\cdots, t_n^q)$.
Then $f$ extends to a toric endomorphism of $X$.
\end{lemma}

\begin{proof}
This is simply because the group homomorphism 
\[
\phi=f_*=q\id_{\mathbf{N}}\colon\mathbf{N}\cong H_1(T_{\mathbf{N}},\mathbb{Z})\to  \mathbf{N}
\]
is always compatible with any fan $\Delta$. 
\end{proof}

\begin{lemma}\label{lem-toric-ext2}
Let $X$ be a normal toric variety with $T$ the big torus of dimension $n$.
Let $g\colon Y\to Z$ be a morphism of varieties.
Let $f\colon T\times Y\to T\times Z$ be a morphism via $f(t_1,\cdots, t_n, y)=(\alpha_1(y)\cdot t_1^q,\cdots, \alpha_n(y)\cdot t_n^q, g(y))$ where $\alpha_i\colon Y\to \mathbf{k}^{\times}$ are morphisms.
Then $f$ extends to a morphism $X\times Y\to X\times Z$.
\end{lemma}

\begin{proof}
By Lemma \ref{lem-toric-ext1}, there is a surjective endomorphism $h\colon X\to X$ such that $$h(t_1,\cdots, t_n)=(t_1^q,\cdots, t_n^q)$$ for any $(t_1,\cdots, t_n)\in T$.
Let $\alpha\colon  Y\to T$ be a morphism such that $\alpha(y)=(\alpha_1(y),\cdots, \alpha_n(y))$.
Define $\overline{f}(x,y)=(\alpha(y)\cdot h(x),g(y))$ for any $(x,y)\in X\times Y$.
Then $\overline{f}$ extends $f$.
\end{proof}

\begin{lemma}\label{lem-toric-power}
Let $f\colon X\to X$ be a surjective endomorphism of a normal projective toric pair $(X,D)$ such that $f^*D_i=qD_i$ for each irreducible component of $D$ and $q\ge 1$.
Let $T$ be the big torus of dimension $n$.
Then $f=\alpha\cdot g$ with $\alpha= f(e_T)$ and $g(t_1,\cdots, t_n)=(t_1^q,\cdots, t_n^q)$ for any $(t_1,\cdots, t_n)\in T$.
\end{lemma}

\begin{proof}
Write $X=T_{\mathbf{N}}(\Delta)$.
Since \(f\) fixes the big torus, it follows that $f=\alpha\cdot g$ with $\alpha= f(e_T)$ and some toric endomorphism $g$.
Hence, $\alpha\cdot D_i=D_i$ and thus $g^*D_i=qD_i$ for each \(i\).
Let $\phi_g\colon\mathbf{N}\to \mathbf{N}$ be the induced lattice endomorphism. 
Since $g$ is finite, by the cone-orbit correspondence, we have $\phi_{g,\mathbb{R}}(\sigma_i)= \sigma_i$ for each $1$-dimensional $\sigma_i\in \Delta$. 
Note that $\phi_{g,\mathbb{R}}|_{\sigma_i}=q\id_{\sigma_i}$.
Since $X$ is projective, the support $|\Delta|=\mathbf{N}_{\mathbb{R}}$ is generated by $1$-dimensional cones.
Therefore, $\phi_{g,\mathbb{R}}=q \id_{\mathbf{N}_{\mathbb{R}}}$ and the lemma is proved.
\end{proof}

\begin{proof}[Proof of Theorem \ref{thm-equi-lifting-general}]
By Lemma \ref{lem-basechange} and Theorem \ref{thm-toric-trivial}, we have the following commutative diagram
\[
\xymatrix{
\widetilde{X}\ar@/^2pc/[rr]^{\tau}\ar[r]^{\sigma}&
\widehat{X}\ar[r]^{\widehat{\pi}}\ar[d]_{p_X}&\widehat{Y}\ar[d]^{p_Y}\\
&X\ar[r]^{\pi}&Y
}
\]
such that \(\widehat{\pi}\colon(X,\widehat{D}\coloneqq p_X^{-1}(D))\to \widehat{Y}\) is  a splitting toric fibration and \(\widetilde{\pi}\colon (X,\widetilde{D}\coloneqq\sigma^{-1}(\widehat{D}))\to \widehat{Y}\) is a splitting smooth toric fibration.
Hence, we can take \(\widetilde{Y}\) to be \(\widehat{Y}\) and replace \((X,D)\) by \((\widehat{X},\widehat{D})\).
Then we may assume that \((X,D)\) is a splitting toric fibration.
Consider the following commutative diagram
$$\xymatrix{
(\widetilde{X},\widetilde{D})\ar[rr]^{\sigma}\ar[rd]_{\widetilde{\pi}}&& (X,D)\ar[ld]^{\pi}\save[]+<2.5pc,0.05pc>*{\racts f} \restore\\
&Y&\save[]+<-3.2pc,0.05pc>*{\racts g} \restore
}$$
satisfying:
\begin{itemize}
\item $\pi$ and $\widetilde{\pi}$ are splitting toric fibrations, 
\item $\sigma$ is a birational morphism with $\sigma^{-1}(D)=\widetilde{D}$, and
\item $f$ is a surjective endomorphism such that $f^*D_i=qD_i$ for each irreducible component of $D$ and $q\ge 1$. (We do iteration here.)
\end{itemize}
Note that $\sigma$ induces a birational morphism $\sigma_0\colon\widetilde{X}_y\to X_y$ such that $\sigma_0^{-1}(D|_{X_y})=\widetilde{D}|_{\widetilde{X}_y}$. 
It is clear that $f$ lifts to a dominant self-map $\widetilde{f}$ on $\widetilde{X}$.
It suffices to show that $\widetilde{f}$ is well-defined everywhere. 
Pick any point \(y\in Y\). 
Let $(\widetilde{X}_y,\widetilde{D}|_{\widetilde{X}_y})$ and $(X_y,D|_{X_y})$ be the fiber of $\widetilde{\pi}$ and $\pi$ over \(y\), respectively.
Since $\pi$ is splitting, $D_{i,y}\coloneqq D_i\cap X_y$ is irreducible for any $y\in Y$ and each $i$.
Choose any two non-empty Zariski open subsets $U_1, U_2\subseteq Y$ such that $g(U_1)\subseteq U_2$, $(\pi^{-1}(U_j), D|_{\pi^{-1}(U_j)})\cong (X_y\times U_j,D|_{X_y}\times U_j)$, and $(\widetilde{\pi}^{-1}(U_j), \widetilde{D}|_{\widetilde{\pi}^{-1}(U_j)})\cong (\widetilde{X}_y\times U_j,\widetilde{D}|_{\widetilde{X}_y}\times U_j)$ for $j=1,2$.
Denote by $S$ the space of surjective endomorphisms \(\varphi\) of $X$. 
So under these isomorphisms, $f$ induces a morphism $f_0\colon X_y\times U_1\to X_{y}\times U_2$ such that $f_0(x, y)=(h(y)(x),g(y))$ where $h\colon U_1\to S$ is a morphism, noting that  
 $U_1\cap U_2\neq\emptyset$ and $h(y)^*(D_i|_{X_y})=qD_i|_{X_y}$ for any \(y\in U_1\).  
Let $T$ be the big torus of $(X_y,D|_{X_y})$ and $e_T=(1,\cdots, 1)\in T$.
Let $\alpha\colon U_1\to T$ be a morphism via $\alpha(y)=h(y)(e_T, y)$. 
By Lemma \ref{lem-toric-power},  $h(y)(t_1,\cdots,t_n)=\alpha(y)\cdot (t_1^q,\cdots, t_n^q)$ for any $(t_1,\cdots,t_n)\in T$ and $y\in U_1$.
So $\sigma_0|_{\widetilde{X}_y\backslash \widetilde{D}|_{\widetilde{X}_y}}\colon \widetilde{X}_y\backslash \widetilde{D}|_{\widetilde{X}_y}\to X_y\backslash D|_{X_y}$ is isomorphic.
In particular, $\widetilde{X}_y$ and $X_y$ share the same big torus $T$.
By Lemma \ref{lem-toric-ext2}, $f_0|_{T\times U_1}\colon T\times U_1\to T\times U_2$ extends to a morphism $\widetilde{X}_y\times U_1\to \widetilde{X}_y\times U_2$.
In particular, $\widetilde{f}$ is well-defined.

Finally, since each fiber is projective and $f$ is surjective, $\widetilde{f}$ is surjective.
\end{proof}

\begin{remark}
It is worth noting that the condition \(f^*D = qD\) in Theorem \ref{thm-equi-lifting-general} cannot be removed, as we are currently unable to extend Lemma \ref{lem-toric-power} to the general case.
\end{remark}

\begin{proof}[Proof of Theorem \ref{thm-equi-lifting}]
Since $\pi$ is already splitting, we see directly from the proof of Theorem \ref{thm-equi-lifting-general}, a birational morphism will be enough.
Note that \(\widetilde{X}\), as a smooth toric fibration over an abelian variety, is also smooth.  
\end{proof}

At the end of this section, we provide the following remark, which is well-known to experts.

\begin{remark}\label{rmk-trivial-mmp}
Let \(\pi\colon(X,B)\to Y\) be a smooth splitting toric fibration over an abelian variety \(Y\); in particular, \(X\) is smooth and \(K_X\) is not pseudo-effective over \(Y\).
By \cite[Corollary 1.3.2]{BCHM10}, we can run relative MMP over \(Y\). 
Let $F$ be a fiber of $\pi$.
Consider the relative MMP 
$$\xymatrix{
X=X_1\ar@{-->}[r]& X_2\ar@{-->}[r] &\cdots\ar@{-->}[r] &X_{m}
}$$
over $Y$.
Let $y\in Y$.
The MMP over $y$ is then a toric MMP.
Let \(B_1=B\) and for each \(2\leq i\leq m\), let \(B_i=\sigma(B_{i-1}^{\nu})\) where $B_{i-1}^{\nu}$ is the sum of irreducible components which are mapped to divisors on \(X_i\).
One can easily verify that {\textbf{\((X_i,B_i)\to Y\) is again a splitting toric fibration over $Y$}}.
\end{remark}

\section{Proof of Main Theorem}\label{sec-pf-main-thm}

In this section, we prove Main Theorem of this paper.

\begin{proof}[Proof of Main Theorem]
By Theorem \ref{theorem-tir}, it suffices for us to prove KSC for \hyperlink{Case TIR}{Case TIR$_n$}.
We prove it by induction on $n$.

By Theorems \ref{thm-tir} and \ref{thm-equi-lifting}, we can take a composition of a quasi-\'etale cover and a resolution to obtain the diagram:
$$\xymatrix{
X&\widetilde{X}\ar[l]_{\sigma}\ar[r]^{\alb_{\widetilde{X}}}& A
}$$
satisfying the following:
\begin{itemize}
\item $\sigma$ is a generically finite surjective morphism,
\item $K_{\widetilde{X}}+\widetilde{D}=\sigma^*(K_X+D)\sim_{\Q} 0$ (cf.~Theorem \ref{thm-toric-trivial}),
\item the Albanese map $\alb_{\widetilde{X}}$ is a \textbf{splitting smooth} toric fibration of the pair $(\widetilde{X},\widetilde{D})$ with $\widetilde{D}=\sigma^{-1}(D)$ (and in particular, \(\widetilde{X}\) is smooth),
\item Some iteration of \(f\) can lift to $\widetilde{f}\in \SEnd(\widetilde{X},\widetilde{D})$, and 
\item \(\SEnd(\widetilde{X},\widetilde{D})\) contains an int-amplified endomorphism \(\widetilde{\mathcal{I}}\).
\end{itemize}

By \cite[Theorem 1.2]{MZ20}, after iteration, we have $\widetilde{f}$-equivariant and $\widetilde{\mathcal{I}}$-equivariant MMP over $A$:
$$\xymatrix{
\widetilde{X}=\widetilde{X}_1\ar@{-->}[r]^{\widetilde{\pi}_1}& \widetilde{X}_2\ar@{-->}[r]^{\widetilde{\pi}_2} &\cdots\ar@{-->}[r]^{\widetilde{\pi}_{s-1}} &\widetilde{X}_{s}\ar[r]^{\widetilde{\tau}} &\widetilde{Y}
}$$
where $\widetilde{\pi}_i$ is birational and $\widetilde{\tau}$ is a Fano contraction of some $K_{\widetilde{X}_s}$-negative extremal ray.
Let $\widetilde{f}_s=\widetilde{f}|_{\widetilde{X}_s}$.
Let $\widetilde{D}_s$  be the image of $\widetilde{D}$ in $\widetilde{X}_s$.
By Remark \ref{rmk-trivial-mmp}, $\widetilde{\tau}$ is again a \textbf{splitting} toric fibration of the pair $(\widetilde{X}_s, \widetilde{D}_s)$.
Note that $-K_{\widetilde{X}_s}\sim_{\mathbb{Q}}\widetilde{D}_s$ and $\widetilde{D}_s$ is reducible, the number of irreducible components of which is greater than \(\dim\widetilde{X}-\dim A\). 
So $(\widetilde{f}_s, \widetilde{X}_s, \widetilde{\tau})$ does not satisfy \hyperlink{Case TIR}{Case TIR} (condition A2) and hence the MMP starting from $\widetilde{X}$ involves at most \hyperlink{Case TIR}{Case TIR$_m$} with $m<n$. 
By Theorem \ref{theorem-tir} and induction, KSC holds for $\widetilde{f}$.
So KSC holds for $f$ by Lemma \ref{lem-ksc-iff}.
\end{proof}

\begin{remark}\label{rem-non-existence-tir}
    Indeed, in the proof of Theorem \ref{main-theorem-ksc}, if we continue to run (any) MMP starting from $\widetilde{Y}$, then it will eventually end up with $A$ and the whole MMP involves no \hyperlink{Case TIR}{Case TIR}, because we are always in the setting of splitting toric fibrations by Remark \ref{rmk-trivial-mmp}. 
    In particular, any MMP starting from \(\widetilde{X}\) does not have \hyperlink{Case TIR}{Case TIR}. 
    However, this does not mean that we can show the non-existence of \hyperlink{Case TIR}{Case TIR} for the initial  \(X\).
\end{remark}

\end{document}